\documentclass[a4paper,12pt]{article}
\usepackage{latexsym,amsmath,amsthm,amssymb}
\usepackage{a4wide}
\usepackage{hyperref}
\usepackage{marginnote}
\usepackage{color}
\usepackage{cite}
\hypersetup{
pdftitle={Entire space}   
pdfauthor={Van Hoang Nguyen},
colorlinks = true,
linkcolor = magenta,
citecolor = blue,
}

\theoremstyle{plain}
\newtheorem{theorem}{Theorem}[section]

\newtheorem{proposition}[theorem]{Proposition}
\newtheorem{lemma}[theorem]{Lemma}

\theoremstyle{definition}

\theoremstyle{remark}

\renewcommand{\thefootnote}{\arabic{footnote}}
\newcommand{\C}[1]{\ensuremath{{\mathcal C}^{#1}}} 

\def\R{\mathbb R}
\def\C{\mathbb C}

\def\al{\alpha}
\def\om{\omega}
\def\Om{\Omega}
\def\be{\beta}
\def\ga{\gamma}
\def\de{\delta}
\def\De{\Delta} 

\def\lam{\lambda}
\def\vphi{\varphi}
\def\ep{\epsilon}
\def\na{\nabla}
\def\pa{\partial}
\def\lt{\left}
\def\rt{\right}
\def\o{\overline}

\numberwithin{equation}{section}

\title{Extremal functions for the Moser--Trudinger inequality of Adimurthi--Druet type in $W^{1,N}(\R^N)$}
\author{Van Hoang Nguyen\footnote{
Institut de Math\'ematiques de Toulouse, Universit\'e Paul Sabatier, 118 Route de Narbonne, 31062 Toulouse c\'edex 09, France.}
}

\begin{document}
\maketitle


\renewcommand{\thefootnote}{}

\footnote{Email: \href{mailto: Van Hoang Nguyen <van-hoang.nguyen@math.univ-toulouse.fr>}{van-hoang.nguyen@math.univ-toulouse.fr}}

\footnote{2010 \emph{Mathematics Subject Classification\text}: 46E35, 26D10.}

\footnote{\emph{Key words and phrases\text}: Moser--Trudinger inequality, blow-up analysis, elliptic estimate, extremal function.}

\renewcommand{\thefootnote}{\arabic{footnote}}
\setcounter{footnote}{0}

\begin{abstract}
We study the existence and nonexistence of maximizers for variational problem concerning to the Moser--Trudinger inequality of Adimurthi--Druet type in $W^{1,N}(\mathbb R^N)$
\[
MT(N,\beta, \alpha) =\sup_{u\in W^{1,N}(\mathbb R^N), \|\nabla u\|_N^N + \|u\|_N^N\leq 1} \int_{\mathbb R^N} \Phi_N(\beta(1+\alpha \|u\|_N^N)^{\frac1{N-1}} |u|^{\frac N{N-1}}) dx,
\]
where $\Phi_N(t) =e^{t} -\sum_{k=0}^{N-2} \frac{t^k}{k!}$, $0\leq \alpha < 1$ both in the subcritical case $\beta < \beta_N$ and critical case $\beta =\beta_N$ with $\beta_N = N \omega_{N-1}^{\frac1{N-1}}$ and $\omega_{N-1}$ denotes the surface area of the unit sphere in $\mathbb R^N$. We will show that $MT(N,\beta,\alpha)$ is attained in the subcritical case if $N\geq 3$ or $N=2$ and $\beta \in (\frac{2(1+2\alpha)}{(1+\alpha)^2 B_2},\beta_2)$ with $B_2$ is the best constant in a Gagliardo--Nirenberg inequality in $W^{1,2}(\mathbb R^2)$. We also show that $MT(2,\beta,\alpha)$ is not attained for $\beta$ small which is different from the context of bounded domains. In the critical case, we prove that $MT(N,\beta_N,\alpha)$ is attained for $\alpha\geq 0$ small enough. To prove our results, we first establish a lower bound for $MT(N,\beta,\alpha)$ which excludes the concentrating or vanishing behaviors of their maximizer sequences. This implies the attainability of $MT(N,\beta,\alpha)$ in the subcritical case. The proof in the critical case is based on the blow-up analysis method. Finally, by using the Moser sequence together the scaling argument, we show that $MT(N,\beta_N,1) =\infty$. Our results settle the questions left open in \cite{doO2015,doO2016}.
\end{abstract}

\section{Introduction}
It is Let $\Om$ be a bounded smooth domain in $\R^N$, $N\geq 2$. The Sobolev inequality asserts that the embedding $W^{1,p}_0(\Om)\hookrightarrow L^{\frac{Np}{N-p}}(\Om)$ holds if $1 \leq p < N$ where $W^{1,p}_0(\Om)$ and $L^p(\Om)$ denote the usual Sobolev space and $L^p$ space on $\Om$ respectively. Such inequality plays an important role in many branches of mathematics such as analysis, geometric, partial differential equations, calculus of variations, etc. In the critical case $p =N$, it is well known that the embedding $W^{1,N}_0(\Om) \hookrightarrow L^\infty(\Om)$ does not hold. In this case, the Moser--Trudinger inequality is a perfect replacement. The Moser--Trudinger inequality was proved independently by Yudovi$\check{\text{\rm c}}$ \cite{Y1961}, Poho$\check{\text{\rm z}}$aev \cite{P1965} and Trudinger \cite{T1967} which asserts the existence of a number $\beta_0$ such that 
\begin{equation}\label{eq:usualMT}
\sup_{u\in W^{1,N}_0(\Om), \|\na u\|_N \leq 1} \int_{\Om} e^{\beta |u|^{\frac{N}{N-1}}} dx < \infty,
\end{equation}
for any $\beta \leq \beta_0$, here $\|\na u\|_N = \lt(\int_\Om |\na u|^N dx\rt)^{1/N}$ denote the usual $L^p$ norm of distributional gradient of $u$. Latter, Moser \cite{M1970} sharpend this inequality by finding the best constant $\beta_0$. More precisely, he proved that the supremum in \eqref{eq:usualMT} is finite for any $\beta\leq \beta_N := N \om_{N-1}^{1/({N-1})}$ with $\om_{N-1}$ denotes the surface are of the unit sphere in $\R^N$, and is infinite if $\beta > \beta_N$. The sharp Moser--Trudinger inequality is a crucial tool in studying the partial differential equation with exponential nonlinearity. Because of its importance, the sharp Moser--Trudinger inequality was generalized to the Heisenberg groups, complex sphere, Riemannian compact manifolds, and hyperbolic space \cite{CLhei,CLcomplex,Li2001,WY2012}.

Suggested by the concentration--compactness principle due to Lions \cite{Lions1985}, the following improvement of the sharp Moser--Trudinger inequality was proposed by Adimurthi and Druet \cite{AD2004} for $N=2$, and by Yang \cite{Yang06*,Yang07corrige} for $N\geq 3$,
\begin{equation}\label{eq:ADY}
\sup_{u\in W^{1,N}_0(\Om), \|\na u\|_N \leq 1} \int_\Om e^{\beta_N (1+ \al \|u\|_N^N)^{\frac1{N-1}} |u|^{\frac N{N-1}}} dx < \infty
\end{equation} 
for any $\al < \lam_1(\Om) =: \inf\{\|\na u\|_N^N: u\in W^{1,N}_0(\Om), \|u\|_N\leq 1\}$. Moreover, the supremum in \eqref{eq:ADY} will be infinite if $\al \geq \lam_1(\Om)$. We refer the reader to the paper \cite{Yang09} for a generalization of this result to compact Riemannian manifold, and to \cite{Tintarev} for an improvement of \eqref{eq:ADY} in dimension two.

The Moser--Trudinger inequality was extended to unbounded domains by Adachi and Tanaka \cite{Adachi}, Cao \cite{Cao}, do \'O \cite{doO97}, Ruf \cite{Ruf2005} and Li and Ruf \cite{LiRuf2008}, namely
\begin{equation}\label{eq:MTfull}
\sup_{u\in W^{1,N}(\R^N), \|u\|_{W^{1,N}(\R^N)} \leq 1} \int_{\R^N} \Phi_N(\be |u|^{\frac N{N-1}}) dx < \infty,
\end{equation}
for any $\beta \leq \beta_N$ where $\|u\|_{W^{1,N}(\R^N)} = \lt(\|\na u\|_N^N + \|u\|_N^N\rt)^{1/N}$ denotes the full Sobolev norm on $W^{1,N}(\R^N)$ and 
\[
\Phi_N(t) =e^t - \sum_{k=0}^{N-2} \frac{t^k}{k!}.
\]
The supremum in \eqref{eq:MTfull} will be infinite if $\beta > \beta_N$. The sharp Moser--Trudinger inequality on entire space was then extended to a singular version general cases by Adimurthi and Yang \cite{AY10}, and to entire Heisenberg group by Lam and Lu \cite{LamLuhei}.

In recent paper \cite{doO2014}, do \'O, de Souza, Medeiros, and Severo proved a result on the weak compactness of the Moser--Trudinger functional defined in $W^{1,N}(\R^N)$, which is a version of the concentration--compactness principle due to Lions for unbounded domains (see also \cite{Nguyen2016} for the analogue result for Adams inequality \cite{Adams1988} which is the version of higher order of derivative of Moser--Trudinger inequality). More precisely, if $u_n$ is a sequence in $W^{1,N}(\R^N)$ such that $\|u_n\|_{W^{1,N}(\R^N)} =1$, $u_n \rightharpoonup u_0$ weakly in $W^{1,N}(\R^N)$, and $u_0\not\equiv 0$, then the inequality \eqref{eq:MTfull} can be improved along the sequence $u_n$ by a constant larger than $\beta_N$, i.e., for any $1< p < (1-\|u_0\|_{W^{1,N}(\R^N)}^N)^{-1/(N-1)}$, we have
\begin{equation}\label{eq:doOcc}
\sup_{n\geq 1} \int_{\R^N} \Phi_N(\beta_N p |u_n|^{\frac N{N-1}}) dx < \infty.
\end{equation}
Suggested by the concentration--compactness type principle \eqref{eq:doOcc}, do \'O and de Souza proved in \cite{doO2014*} for $N=2$ and in \cite{doO2015} for $N\geq 3$ the following Moser--Trudinger type inequality
\begin{equation}\label{eq:doOdeSouza}
\sup_{u\in W^{1,N}(\R^N), \|u\|_{W^{1,N}(\R^N)} \leq 1} \int_{\R^N}\Psi_N(\beta_N (1+\al \|u\|_N^N)^{\frac1{N-1}} |u|^{\frac N{N-1}}) dx < \infty.
\end{equation}
for any $0\leq \al < 1$, where 
\[
\Psi_N = e^t -\sum_{k=0}^{N-1} \frac{t^k}{k!} = \Phi_N(t) - \frac{t^{N-1}}{(N-1)!}.
\]
As a corollary, do \'O and de Souza obtained the following improved version of \eqref{eq:MTfull} in the spirit of Adimurthi, Druet and Yang in \cite{doO2015,doO2016}
\begin{equation}\label{eq:doOdeSouza*}
MT(N,\beta,\alpha):=\sup_{u\in W^{1,N}(\R^N), \|u\|_{W^{1,N}(\R^N)} \leq 1} \int_{\R^N}\Phi_N(\beta (1+\al \|u\|_N^N)^{\frac1{N-1}} |u|^{\frac N{N-1}}) dx < \infty.
\end{equation}
for any $\beta \leq \beta_N$ and $0\leq \al < 1$. When $\alpha =0$, \eqref{eq:doOdeSouza*} reduces to \eqref{eq:MTfull}, and for $\al >0$, \eqref{eq:doOdeSouza*} is stronger than \eqref{eq:MTfull}. We mention here that a similar result to \eqref{eq:doOdeSouza*} was proved by do \'O and de Souza in \cite{doO2014*} on the whole plane for the subspace of $W^{1,2}(\R^2)$
\[
E_V =\lt\{u\in W^{1,2}(\R^2): \int_{\R^2} V(x) u^2(x) dx < \infty\rt\},
\]
where $V$ is nonnegative, bounded away from zero, radially increasing and coercive, i.e., $V(x) \to \infty$ as $|x|\to \infty$. These assumptions on $V$ ensure the compactness of the embedding $E_V \hookrightarrow L^s(\R^2)$ for any $2 \leq s < \infty$ which plays an important role in the proof in \cite{doO2014*}. 

As usual, the proof of \eqref{eq:doOdeSouza} given in \cite{doO2015,doO2016} is based on the blow-up analysis method. However, it is easy to see that it can be deduced from \eqref{eq:MTfull}. Indeed, for any $\tau >0$, by scaling argument $u_\tau(x) = u(\tau^{1/N} x)$, we get that
\begin{align*}
C_\tau =\sup_{u\in W^{1,N}(\R^N), \|\na u\|_N^N +\tau \|u\|_N^N \leq 1} &\int_{\R^N}\Phi_N(\beta_N  |u|^{\frac N{N-1}}) dx\\
&=\frac1{\tau} \sup_{u\in W^{1,N}(\R^N), \|u\|_{W^{1,N}(\R^N)} \leq 1} \int_{\R^N}\Phi_N(\beta_N  |u|^{\frac N{N-1}}) dx = \frac{C_1}{\tau},
\end{align*}
thus is finite by \eqref{eq:MTfull}. Choosing $\tau = 1 -\alpha$, and for $u\in W^{1,N}(\R^N)$ with $\|u\|_{W^{1,N}(\R^N)} \leq 1$, define
\[
w =\frac{u}{(\|\na u\|_N^N + \tau \|u\|_N^N)^{\frac1N}}.
\]
We then have $\|\na w\|_N^N + \tau \|w\|_N^N =1$. From the observation above, we have  
\begin{equation}\label{eq:ap}
\int_{\R^N}\Phi_N(\beta_N  |w|^{\frac N{N-1}}) dx \leq C_\tau = \frac{C_1}{1-\al}.
\end{equation}
Remark that $|u|^{\frac{N}{N-1}} \leq (1 -\alpha \|u\|_N^N)^{\frac1{N-1}} |w|^{\frac N{N-1}}$ by the choice of $\tau$. Hence, we get
\[
(1+ \al \|u\|_N^N)^{\frac1{N-1}} |u|^{\frac{N}{N-1}} \leq |w|^{\frac N{N-1}},
\]
which together \eqref{eq:ap} implies \eqref{eq:doOdeSouza*}.

An interesting problem on the Moser--Trudinger inequality is whether extremal functions exist or not. Existence of extremal functions for the Moser--Trudinger inequality \eqref{eq:usualMT} was proved by Carleson and Chang \cite{CC1986} when $\Om$ is the unit ball, by Struwe \cite{Struwe} when $\Om$ is close to the ball in the sense of measure, by Flucher \cite{Flucher1992} and Lin \cite{Lin1996} when $\Om$ is a general smooth bounded domain, and by Li \cite{Li2001} for compact Riemannian surfaces. For recent developments, we refer the reader to \cite{Li2005,Yang06,Yang06*,Yang09,Yang2015,YZ}. The existence of extremal functions for the Moser--Trudinger inequality \eqref{eq:MTfull} was studied by Ruf \cite{Ruf2005} for $N=2$ and $\beta =\beta_2$, by Li and Ruf \cite{LiRuf2008} for $N\geq 3$ and $\beta =\beta_N$, and by Ishiwata \cite{Ishi} for $N=2$, $\beta \leq \beta_2$ and $N\geq 3$, $\beta < \beta_N$. The existence results in \cite{Ruf2005,LiRuf2008,Ishi} say that the extremal functions for \eqref{eq:MTfull} exist for $N\geq 3$, $\beta \leq \beta_N$ and for $N=2$, $\ep_0 \leq \beta \leq \beta_2$, for some $\ep_0 \in (0,\beta_2)$. Moreover, it was proved by Ishiwata \cite{Ishi} in dimension two that for $\beta >0$ sufficiently small, the extremal function for \eqref{eq:MTfull} do not exist. We refer reader to \cite{LY2016} for more recent result on the existence of extremal functions for the singular Moser--Trudinger inequality in whole space $\R^N$, $N\geq 2$.

Concerning to the extremal functions of \eqref{eq:doOdeSouza} and \eqref{eq:doOdeSouza*}, it was proved in \cite{doO2015,doO2016} that extremal functions for \eqref{eq:doOdeSouza} exist for any $0\leq \alpha < 1$. Note that $\Psi_N$ is obtained from $\Phi_N$ by subtracting the term corresponding to the $L^N$ norm. This modification allows us to gain the compactness necessary of the maximizing sequence to prove the the attainability of the supremum in \eqref{eq:doOdeSouza}. The question on the extremal functions of \eqref{eq:doOdeSouza*} was left open in \cite{doO2015,doO2016}. Our main aim of this paper is to settle this question. Moreover, we study the existence of extremal functions for \eqref{eq:doOdeSouza*} both in the subcritical case $\beta < \beta_N$ and in the critical case $\beta=\beta_N$. Note that the existence of extremal functions for the subcritical case is also nontrivial since the problem suffers from the lack of compactness due to the unboundedness of domain. To state our first main result, let us denote
\begin{equation}\label{eq:B2def}
B_2 = \sup_{\phi \in W^{1,2}(\R^2), \phi\not\equiv 0} \frac{\|\phi\|_4^4}{\|\na \phi\|_2^2 \|\phi\|_2^2}
\end{equation}
the best constant in the Gagliardo--Nirenberg inequality in $W^{1,2}(\R^2)$. It is known that $B_2$ is attained and $B_2 > 1/(2\pi)$ (see, e.g., \cite{Beckner,Weinstein}). Our first main result of this paper reads as follows

\begin{theorem}\label{Maintheorem}
Let $N \geq 2$ and $0\leq \alpha < 1$. There exists $u\in W^{1,N}(\R^N)$ such that $\|u\|_{W^{1,N}(\R^N)} =1$ and 
\[
MT(N,\beta,\al) = \int_{\R^N} \Phi_N (\beta(1+\al \|u\|_N^N)^{\frac1{N-1}} |u|^{\frac N{N-1}}) dx,
\]
in the following cases:
\begin{description}
\item (i) \emph{\bf(subcritical case)\text} For any $0\leq \al < 1$ and $\beta < \beta_N$ if $N\geq 3$ and $\frac{2(1+2\al)}{(1+\al)^2 B_2} < \beta < \beta_2$ if $N =2$.
\item (ii) \emph{\bf(critical case)} For $\beta =\beta_N$ and for any $0\leq \al < \al_0\in (0,1)$ for some $\al_0 \in (0,1)$.
\end{description}
Moreover, $MT(2,\beta,\al)$ is not attained if $\beta \ll 1$ for any $0 \leq \al < 1$.
\end{theorem}
Since $B_2 > 1/(2\pi)$ and $\beta_2 = 4\pi$ then our assumption of Theorem \ref{Maintheorem} makes sense in the dimension two. 

It was shown in \cite{doO2015,doO2016} that $MT(N,\beta_N,\al) =\infty$ for any $\al >1$. It was also asked in those papers that $MT(N,\beta_N,1)$ is finite or not as an open problem. In this paper, by using the Moser sequence \cite{M1970,Adachi} together the scaling argument, we will prove that $MT(N,\beta_N,1) =\infty$. Thus $1$ is the critical threshold of $\al$ for which $MT(N,\beta,\al)$ is finite.

\begin{theorem}\label{Maintheorem2}
For any $N \geq 2$, it holds $MT(N,\beta_N,1) =\infty$.
\end{theorem}

Let us give the outline of the proof of our main Theorems. By a rearrangement argument, we can restrict ourselves to the decreasing radially, symmetric, nonnegative functions in $W^{1,N}(\R^N)$ to prove Theorem \ref{Maintheorem}. The proof in the subcritical case $\beta < \beta_N$ is inspired by the recent paper of Ishiwata \cite{Ishi}. Following the argument of Ishiwata, we prove an useful lower bound for $MT(N,\beta,\alpha)$ (see Proposition \ref{lowerboundsub} below) which excludes the concentration or the vanishing behavior of maximizing sequences in the subcritical case. From this, we obtain the attainability of $MT(N,\beta,\al)$ for $\beta < \beta_N$. The attainability of $MT(N,\beta_N,\al)$ is proved by the blow--up analysis method as done in \cite{doO2015,doO2016} for the Moser--Trudinger type inequality concerning to function $\Psi_N$. For more about the blow--up analysis method, we refer the reader to the book \cite{Druet} and to the papers \cite{AD2004,CC1986,doO2014*,doO2015,doO2016,Li2001,Li2005,LiRuf2008,LY2016,WY2012,Yang06,Yang06*,Yang09,Yang2015,YZ}. To prove Theorem \ref{Maintheorem2}, as mentioned before, we use the Moser sequence together the scaling argument.

Let us mention here that after this work was done and submitted to arXiv, the author was informed the work of Lu and Zhu \cite{LuZhu} (also on arXiv) in which they proved the statement $(ii)$ of Theorem \ref{Maintheorem} and Theorem \ref{Maintheorem2}. Although these two works are based on the blow-up analysis method, however there is still difference in our proofs. Indeed, in this paper, the author proves the existence of the extremal functions for the subcritical inequalities in the entire space $\R^N$ (this result is nontrivial as mentioned above), and then performs the blow-up process in the entire space $\R^N$. The work of Lu and Zhu \cite{LuZhu} follows the traditional strategy in \cite{LiRuf2008,doO2015}. They first prove the existence of extremal functions for the subcritical inequality in balls centered at origin with radius tending to infinity, and then perform the blow-up process for the sequence of these extremal functions. Note that the existence of extremal functions for the subcritical inequality in the ball is more easily proved than the one in the entire space. Moreover, it seems that there is an incomplete proof in the proof of Lemma $4.3$ in \cite{LuZhu} to exclude the case where the weak limit $u$ of $u_k$ is zero function. Indeed, in \cite{LuZhu}, Lu and Zhu only excluded the situation where $u_k$ is a Sobolev-normalized concentrating sequence. However, there is other situation where $\|u_k\|_N \not\to 0$, for example, when $u_k$ is normalized vanishing sequence (see \cite{Ishi}). In this situation, the conclusion $\int_{\R^n} \Phi(\alpha_k u_k^{\frac n{n-1}}) dx\to \int_{\R^n} \Phi(\alpha_n u^{\frac n{n-1}}) dx = 0$ of Lu and Zhu does not hold. This is one of the main difficulties of this problem and is the reason why do \'O and de Souza considered the modification of $\Phi_N$, that is the function $\Psi_N$ above, to gain the compactness in their arguments. This difficulty was overcome by the author in Lemma \ref{boundedcase} below. 

We conclude this introduction by introducing some notation used throughout this paper. For a subset $\Om$ of $\R^N$, $p\geq 1$ and $u$ is a function defined on $\Om$, we use $\|u\|_{p,\Om}$ and $\|\na u\|_{p,\Om}$ to denote the $L^p$ norms of $u$ and $\na u$ with respect to the Lebesgue measure on $\Om$ respectively, i.e., $\|u\|_{p,\Om} =\lt(\int_{\Om} |u|^p dx\rt)^{1/p}$ and $\|\na u\|_{p,\Om} = \lt(\int_\Om |\na u|^p dx\rt)^{1/p}$. When $\Om =\R^N$, we simply denote by $\|u\|_p$ and $\|\na u\|_p$. The full Sobolev norm of a function $u$ in $W^{1,N}(\R^N)$ is defined by $\|u\|_{W^{1,N}(\R^N)} = \lt(\|\na u\|_N^N + \|u\|_N^N\rt)^{1/N}$.

The rest of this paper is organized as follows. In the next sections \S2, we prove Theorem \ref{Maintheorem} in the subcritical case $\beta < \beta_N$, and consequently we obtain a maximizing sequence of decreasing radially symmetric nonnegative functions for $MT(N,\beta_N,\al)$ for $0\leq \al < 1$. We also prove a nonnexistence result in section \S2. In section \S3, we prove the attainability of $MT(N,\beta_N,\al)$ for any $\al >0$ small by using the method of blow-up analysis. In the last section \S4, we use the Moser sequence and the scaling argument to prove that $MT(N,\beta_N,1) =\infty$ for any $N\geq 2$.

\section{The subcritical case and a nonexistence result in dimension two}
In this section, we consider the case $\beta < \beta_N$. We note that the existence result in this subcritical case is nontrivial since the lacking of compactness due to the unboundedness of the domain of this problem. Indeed, due to this difficulty in the subcritical case, the nonexistence occurs for small $\beta$ when $N=2$. 

Let $\{u_n\}_n\subset W^{1,N}(\R^N)$ be a maximizing sequence for $MT(N,\beta,\al)$. By P\'olya--Szeg\"o principle \cite{Brothers}, we can assume that $u_n$ is decreasing radially symmetric, nonnegative  around. Abusing of notation, we will write $u_n(r)$ for $u_n(x)$ with $r =|x|$ for simplifying notation. By Rellich--Kondrachov theorem, we can assume in addition that $u_n \rightharpoonup u_0$ weakly in $W^{1,N}(\R^N)$, $u_n\to u_0$ in $L^p_{\rm loc}(\R^N)$ for any $p < \infty$ and $u_n\to u_0$ a.e. in $\R^N$.

We first prove a lower bound for $MT(N,\beta,\alpha)$. This bound plays a crucial role in our analysis below. In deed, it excludes the concentration and vanishing behavior of the maximizing sequence $u_n$.

\begin{proposition}\label{lowerboundsub}
Let $N\geq 2$, $\alpha \in [0,1)$ and $\beta \in (0,\beta_N]$. It holds that 
\[
MT(N,\be,\al) > \frac{\beta^{N-1}}{(N-1)!} (1+\alpha)
\]
for $\beta \in (0,\beta_N]$ when $N\geq 3$ and for $\beta \in (\frac{2(1+2\alpha)}{(1+\al)^2B_2},\beta_2]$ when $N=2$ where $B_2$ is given by
\[
B_2 = \sup_{v\in W^{1,2}(\R^2), v\not=0} \frac{\|v\|_4^4}{\|\na v\|_2^2 \|v\|_2^2}.
\]
\end{proposition}
\begin{proof}
We follow the argument of Ishiwata \cite{Ishi}. For $v\in W^{1,N}(\R^N)$ and $t >0$, we introduce a family of functions $v_t$ by
\[
v_t(x) = t^{\frac1N} v(t^{\frac1N}x).
\]
We can easily check that 
\[
\|\na v_t\|_N^N = t\|\na v\|_N^N,\qquad \|v_t\|_p^p = t^{\frac{p-N}N} \|v\|_N^N,
\]
for $p\geq N$. Using this relation, and the inequality $\Phi_N(t) \geq \frac{\beta^{N-1}}{(N-1)!} t^{N-1} + \frac{\beta^N}{N!} t^N$ we have
\begin{align*}
MT(N,\be,\al) &\geq \int_{\R^N} \Phi_N\lt(\beta\lt(1+ \al \frac{\|v_t\|_N^N}{\|v_t\|_{W^{1,N}(\R^N )}^N}\rt)^{\frac1{N-1}} \frac{v_t^{\frac N{N-1}}}{\|v_t\|_{W^{1,N}(\R^N )}^{\frac N{N-1}}}\rt) dx\\
&\geq \frac{\beta^{N-1}}{(N-1)!} \lt(1+ \frac{\al \|v\|_N^N}{t\|\na v\|_N^N +\|v\|_N^N}\rt) \frac{\|v\|_N^N}{t\|\na v\|_N^N +\|v\|_N^N} \\
&\quad + \frac{\beta^N}{N!} \lt(1+ \frac{\al \|v\|_N^N}{t\|\na v\|_N^N +\|v\|_N^N}\rt)^{\frac{N}{N-1}} \frac{t^{\frac1{N-1}} \|v\|_{\frac{N^2}{N-1}}^{\frac{N^2}{N-1}}}{(t\|\na v\|_N^N +\|v\|_N^N)^{\frac N{N-1}}}.
\end{align*} 
It is easily to see that
\begin{align*}
\frac{\beta^{N-1}}{(N-1)!} \lt(1+ \frac{\al \|v\|_N^N}{t\|\na v\|_N^N +\|v\|_N^N}\rt)& \frac{\|v\|_N^N}{t\|\na v\|_N^N +\|v\|_N^N}\\
& =\frac{\beta^{N-1}}{(N-1)!} \lt(1+\al\rt) -t\frac{\beta^{N-1}}{(N-1)!} \lt(1+2\al\rt)\frac{\|\na v\|_N^N}{\|v\|_N^N} +o(t)
\end{align*}
and
\begin{align*}
\frac{\beta^N}{N!} \lt(1+ \frac{\al \|v\|_N^N}{t\|\na v\|_N^N +\|v\|_N^N}\rt)^{\frac{N}{N-1}}& \frac{t^{\frac1{N-1}} \|v\|_{\frac{N^2}{N-1}}^{\frac{N^2}{N-1}}}{(t\|\na v\|_N^N +\|v\|_N^N)^{\frac N{N-1}}}\\
& = \frac{\beta^N}{N!}(1+\al)^{\frac N{N-1}} \frac{\|v\|_{\frac {N^2}{N-1}}^{\frac{N^2}{N-1}}}{\|v\|_N^{\frac {N^2}{N-1}}} t^{\frac 1{N-1}} + O(t^{\frac N{N-1}}).
\end{align*}
as $t\to 0$. Thus we get the conclusion when $N\geq 3$ by choosing $t>0$ small enough. In the case $N =2$, we have from estimates above that
\[
MT(N,\beta,\al) \geq \beta(1+\al) + \frac{\beta}2 \lt(\beta -\frac{2(1+2\al)}{(1+\alpha)^2} \frac{\|\na v\|_2^2 \|v\|_2^2}{\|v\|_4^4} \rt)\frac{\|v\|_2^4}{\|v\|_4^4} t + o(t).
\]
It is well known that $B_2$ is attained by a function $U\in W^{1,2}(\R^2)$ (see \cite{Weinstein}). By taking $v=U$, we obtain the conclusion when $N=2$.
\end{proof}

We next recall the following radial lemma for radial function in $W^{1,N}(\R^N)$.

\begin{lemma}\label{radiallemma}
Let $u\in W^{1,N}(\R^N)$ be a radial function, then
\[
|u^N(r)| \leq \frac{C}{r^{N-1}} \|u\|_{W^{1,N}(\R^N)}
\]
for any $r>0$, where $C$ is constant depending only on $N$.
\end{lemma}
\begin{proof}
Since $u\in L^N(\R^N)$, then there exists $r_n\to \infty$ such that $u(r_n) \to 0$. We have
\[
u(r)^N = u(r_n)^N -N\int_r^{r_n} u(s)^{N-1} u'(s) ds.
\]
Hence
\begin{align*}
|u(r)|^N &\leq |u(r_n)|^N + \int_r^{r_n} |u(s)|^{N-1} |u'(s)|ds \\
&\leq |u(r_n)|^N + r^{-N+1} \int_r^{r_n} |u(s)|^{N-1} s^{\frac{(N-1)^2}N} |u'(s)| s^{\frac{N-1}{N}} ds.
\end{align*}
Letting $n\to \infty$ and applying H\"older inequality, we get
\[
|u(r)|^N \leq \frac{C_1}{r^{N-1}} \|\na u\|_N^{\frac1N} \|u\|_{N}^{\frac{N-1}N},
\]
with $C_1$ depends only on $N$. Using Young's inequality, we get the conclusion.
\end{proof}

With Proposition \ref{lowerboundsub} and Lemma \ref{radiallemma} in hand, we are ready to prove the attainability of $MT(N,\beta,\al)$ for $\beta < \beta_N$. The argument below is inspired by the paper of Ishiwata \cite{Ishi}. However, our argument is simpler than the one in \cite{Ishi}.\\

\begin{proof}[Proof of Theorem \ref{Maintheorem} in the subcritical case]
Since $0< \|u_n\|_N^N \leq 1$, we can assume by subtracting a subsequence that $\|u_n\|_N^N \to a \in [0,1]$. For any $R >1$, let us denote $B_R= \{x: |x| < R\}$ and $B_R^c = \R^N \setminus B_R$. We define the new function $v_{n,R}$ on $B_R$ by 
\[
v_{n,R}(r) = u_n(r) -u_n(R),\quad r < R,
\]
then $v_{n,R}\in W_0^{1,N}(B_R)$ and
\[
\|\na v_{n,R}\|_N^N = \int_{B_R} |\na u_n|^N dx =1-\int_{B_R^c} |\na u_n|^N dx -\int_{\R^N} u_n^N dx  \leq 1 -\|u_n\|_N^N,
\]
and
\begin{align}\label{eq:chantrenun}
u_n(r)^{\frac N{N-1}} &\leq (1+\de) v_{n,R}(r)^{\frac N{N-1}} + (1 -(1+\de)^{\frac1{1-N}})^{1-N} u_n(R)^{\frac{N}{N-1}}\notag\\
&\leq (1+\de) v_{n,R}(r)^{\frac N{N-1}} + (1 -(1+\de)^{\frac1{1-N}})^{1-N} \frac C{R},
\end{align}
for any $\de >0$ and $0< r < R$, here we use Lemma \ref{radiallemma}. 

Denote $w_{n,R} = v_{n,R}/ \|\na v_{n,R}\|_N$, we have from \eqref{eq:chantrenun} that
\begin{align*}
\Phi_N(\beta (1+\al \|u_n\|_N^N)^{\frac1{N-1}} u_n^{\frac N{N-1}}) &\leq e^{\beta (1+\al \|u_n\|_N^N)^{\frac1{N-1}} u_n^{\frac N{N-1}}}\\
&\leq e^{\beta(1+\al)(1 -(1+\de)^{\frac1{1-N}})^{1-N} \frac C{R}} e^{\beta(1+\de) (1+\al \|u_n\|_N^N)^{\frac1{N-1}} v_{n,R}^{\frac N{N-1}}}\\
&\leq e^{\beta(1+\al)(1 -(1+\de)^{\frac1{1-N}})^{1-N} \frac C{R}} e^{\beta(1+\de)  w_{n,R}^{\frac N{N-1}}},
\end{align*}
on $B_R$. Choosing $\de >0$ sufficiently small such that $\beta(1+\de) < \beta_N$. By classical Moser--Trudinger inequality \eqref{eq:usualMT}, we conclude that $e^{\beta(1+\de) w_n^{\frac N{N-1}}}$ is bounded in $L^q(B_R)$ for some $q >1$. Consequently, $\Phi_N(\beta (1+\al \|u_n\|_N^N)^{\frac1{N-1}} u_n^{\frac N{N-1}})$ is bounded in $L^q(B_R)$ for some $q >1$. Since $\|u_n\|_N^N \to a$ and $u_n\to u_0$ a.e., then we have
\begin{equation}\label{eq:trunlimit1}
\lim_{n\to\infty}\int_{B_R} \Phi_N(\beta (1+\al \|u_n\|_N^N)^{\frac1{N-1}} u_n^{\frac N{N-1}})dx = \int_{B_R} \Phi_N(\beta(1+ \al a)^{\frac1{N-1}} u_0^{\frac N{N-1}}) dx.
\end{equation}
Obviously,
\begin{equation}\label{eq:trunlimit2}
\lim_{n\to\infty} \int_{B_R(0)} (1+\al \|u_n\|_N^N) u_n^{N} dx = (1+\al a) \int_{B_R(0)} u_0^N dx.
\end{equation}
In the other hand, using again Lemma \ref{radiallemma}, we have for $r\geq R$
\begin{align*}
\Phi_N(\beta (1+\al \|u_n\|_N^N)^{\frac1{N-1}} &u_n(r)^{\frac N{N-1}}) -\frac{\beta^{N-1}}{(N-1)!} (1+ \al \|u_n\|_N^N) u_n(r)^{N}\\
&=\sum_{k=N}^{\infty} \frac{\beta^k}{k!} (1+ \al \|u_n\|_N^N)^{\frac k{N-1}} u_n(r)^{\frac{kN}{N-1}} \\
&\leq \sum_{k=N}^{\infty} \frac{\beta^k}{k!} (1+ \al)^{\frac k{N-1}} \lt(\frac C{R^{N-1}}\rt)^{\frac{k-N+1}{N-1}} u_n(r)^N\\
&\leq \frac1{CR} e^{\beta (1+\al)^{\frac1{N-1}} C^{\frac 1{N-1}}} u_n(r)^N\\
&=\frac{C'}{R} u_n(r)^N,
\end{align*}
with $C'$ independent of $n$ and $R$. This implies
\begin{equation}\label{eq:outsideBR}
\lim_{R\to\infty}\lim_{n\to\infty} \int_{B_R^c} \Bigg(\Phi_N(\beta (1 +\al \|u_n\|_N^N)^{\frac1{N-1}} u_n^{\frac N{N-1}}) - \frac{\beta^{N-1}}{(N-1)!} (1+ \al \|u_n\|_N^N) u_n^{N}\Bigg)dx =0.
\end{equation}
Since,
\begin{align*}
\int_{\R^N} \Bigg(\Phi_N(&\beta (1 +\al \|u_n\|_N^N)^{\frac1{N-1}} u_n^{\frac N{N-1}}) - \frac{\beta^{N-1}}{(N-1)!} (1+ \al \|u_n\|_N^N) u_n^{N}\Bigg)dx\\
&= \int_{B_R} \Bigg(\Phi_N(\beta (1 +\al \|u_n\|_N^N)^{\frac1{N-1}} u_n^{\frac N{N-1}}) - \frac{\beta^{N-1}}{(N-1)!} (1+ \al \|u_n\|_N^N) u_n^{N}\Bigg)dx\\
&\quad + \int_{B_R^c} \Bigg(\Phi_N(\beta (1 +\al \|u_n\|_N^N)^{\frac1{N-1}} u_n^{\frac N{N-1}}) - \frac{\beta^{N-1}}{(N-1)!} (1+ \al \|u_n\|_N^N) u_n^{N}\Bigg)dx.
\end{align*}
Letting $n\to\infty$ and then $R\to \infty$, and using \eqref{eq:trunlimit1}, \eqref{eq:trunlimit2} and \eqref{eq:outsideBR}, we obtain
\begin{equation}\label{eq:dangthucsup}
MT(N,\beta,\alpha) = \int_{\R^N}\Phi_N( \beta(1+ \al a)^{\frac1{N-1}} u_0^{\frac N{N-1}}) dx +\frac{\beta^{N-1}}{(N-1)!}(1+\alpha a)(a -\|u_0\|_N^N).
\end{equation}
If $u_0 \equiv 0$, then 
\[
MT(N,\beta,\alpha)= \frac{\beta^{N-1}}{(N-1)!}(1+\alpha a)a\leq \frac{\beta^{N-1}}{(N-1)!}(1+\alpha),
\]
which is impossible by Proposition \ref{lowerboundsub}. Hence $u_0\not\equiv0$. Define $\tau^N = a/\|u_0\|_N^N \geq 1$, and $v_0(x) = u_0(x/\tau)$. We have
\[
\|v_0\|_N^N = \|u_0\|_N^N \tau^N = a,\quad\text{ and }\quad \|\na v_0\|_N^N = \|\na u_0\|_N^N.
\]
Hence
\[
1 =\liminf_{n\to\infty} (\|\na u_n\|_N^N + \|u_n\|_N^N)\geq \liminf_{n\to\infty} \|\na u_n\|_N^N + \lim_{n\to\infty} \|u_n\|_N^N \geq \|v_0\|_{W^{1,N}(\R^N)}^N.
\]
Hence
\begin{align*}
MT(N,\beta,\al) &\geq \int_{\R^N} \Phi_N(\beta (1 +\al \|v_0\|_N^N )^{\frac1{N-1}} v_0^{\frac{N}{N-1}}) dx\\
&=\tau^N \int_{\R^N} \Phi_N(\beta (1+\al a)^{\frac1{N-1}} u_0^{\frac N{N-1}}) dx\\
&=\int_{\R^N} \Phi_N(\beta (1+\al a)^{\frac1{N-1}} u_0^{\frac N{N-1}}) dx + (\tau^N-1)\frac{\beta^{N-1}}{(N-1)!}(1+ \al a) \|u_0\|_N^N\\
&\quad +(\tau^N-1) \int_{\R^N}\Bigg(\Phi_N(\beta (1 +\al a)^{\frac1{N-1}} u_0^{\frac N{N-1}}) - \frac{\beta^{N-1}}{(N-1)!} (1+ \al a) u_0^{N}\Bigg) dx\\
&=\int_{\R^N} \Phi_N(\beta (1+\al a)^{\frac1{N-1}} u_0^{\frac N{N-1}}) dx + \frac{\beta^{N-1}}{(N-1)!}(1+ \al a)(a- \|u_0\|_N^N)\\
&\quad +(\tau^N-1) \int_{\R^N}\Bigg(\Phi_N(\beta (1 +\al a)^{\frac1{N-1}} u_0^{\frac N{N-1}}) - \frac{\beta^{N-1}}{(N-1)!} (1+ \al a) u_0^{N}\Bigg) dx\\
&=MT(N,\beta,\al) + (\tau^N-1) \int_{\R^N}\Bigg(\Phi_N(\beta (1 +\al a)^{\frac1{N-1}} u_0^{\frac N{N-1}})\\
&\hspace{8cm} - \frac{\beta^{N-1}}{(N-1)!} (1+ \al a) u_0^{N}\Bigg) dx,
\end{align*}
here we use \eqref{eq:dangthucsup}. Since $u_0\not\equiv 0$ then $\tau =1$, or $a = \|u_0\|_N^N$. Using again \eqref{eq:dangthucsup}, we have
\[
MT(N,\beta,\alpha) = \int_{\R^N}\Phi_N( \beta(1+ \al \|u_0\|_N^N)^{\frac1{N-1}} u_0^{\frac N{N-1}}) dx.
\]
This together the fact $\|u_0\|_{W^{1,N}(\R^N)} \leq 1$ implies $\|u_0\|_{W^{1,N}(\R^N)} =1$ and hence $u_0$ is a maximizer for $MT(N,\beta,\alpha)$.
\end{proof}

We conclude this section by proving the nonexistence of extremal functions for \eqref{eq:doOdeSouza*} in dimension two when $\beta >0$ sufficiently small. We follow the argument in \cite{Ishi}. 

\begin{proof}[Proof of nonexistence result in dimension two]
Let us recall the Moser--Trudinger type inequality in $\R^2$ of Adachi and Tanaka (see \cite{Adachi}),
\[
C_\beta = \sup_{u\in W^{1,2}(\R^2), u\not=0, \|\na u\|_2 \leq 1} \frac{ \|\na u\|_2^2}{\|u\|_2^2} \int_{\R^2} \lt(e^{\beta \frac{u^2}{\|\na u\|_2^2}} -1\rt) dx < \infty,
\]
for any $0 < \beta < 4\pi$ and $C_\be$ depends only on $\beta$. From this we have
\[
\frac{\beta^j}{j!} \frac{\|u\|_{2j}^{2j}}{\|\na u\|_2^{2j}} \leq C_\beta \frac{\|u\|_2^2}{\|\na u\|_2^2},
\]
for any $j \geq 1$, or equivalently
\begin{equation}\label{eq:GN}
\|u\|_{2j}^{2j} \leq C_\be \frac{j!}{\beta^j} \|\na u\|_2^{2(j-1)} \|u\|_2^2,\quad j\geq 1.
\end{equation}
Let $\mathcal M = \{u\in W^{1,2}(\R^2): \|u\|_{W^{1,2}(\R^2)} =1\}$. For any $v\in \mathcal M$, define for $t >0$
\[
v_t(x) = t^{\frac12} v(t^{\frac12} x),\quad w_t = \frac {v_t}{\|v_t\|_{W^{1,2}(\R^2)}}.
\]
Then $w_t \in \mathcal M$. Denote
\[
J[u] = \int_{\R^2} \lt(e^{\beta(1+\al \|u\|_2^2) u^2} -1\rt) dx.
\]
If $v$ is a maximizer for $MT(2,\beta,\al)$ then $v\in \mathcal M$. Note that $w_t$ is a curve in $\mathcal M$, $w_1 =v$, hence
\begin{equation}\label{eq:critical}
\frac{d}{dt}J[w_t]\bigg|_{t=1} =0.
\end{equation}
From the definition of $v_t$, we have $\|v_t\|_p^p = t^{\frac{p-2}2} \|v\|_p^p$, $\|\na v_t\|_2^2 =t\|\na v\|_2^2$, we then have
\begin{align*}
J[w_t] &= \sum_{j=1}^\infty \frac{\beta^j}{j!} \lt(1+\al \frac{\|v\|_2^2}{t\|\na v\|_2^2 + \|v\|_2^2}\rt)^j \frac{t^{j-1}\|v\|_{2j}^{2j}}{(t\|\na v\|_2^2 + \|v\|_2^2)^j}.
\end{align*}
Hence, in view of $\|v\|_{W^{1,2}(\R^2)} =1$ and \eqref{eq:GN} we have
\begin{align*}
\frac{d}{dt} J[w_t] \bigg|_{t=1} &= \sum_{j=1}^\infty \frac{\beta^j}{j!}(1+\al\|v\|_2^2)^{j-1} \|v\|_{2j}^{2j}\Big(-j\al \|v\|_2^2 \|\na v\|_2^2 \\
&\hspace{6cm}+(j-1) (1+\al \|v\|_2^2) -j\|\na v\|_2^2(1+ \al \|v\|_2^2\Big)\\
&\leq-\beta \|v\|_2^2 \|\na v\|_2^2 + \sum_{j\geq 2} \frac{(2\beta)^j}{(j-1)!}  \|v\|_{2j}^{2j}\\
&\leq \|v\|_2^2 \|\na v\|_2^2\lt(-\beta + \sum_{j\geq 2} \frac{(2\beta)^j}{(j-1)!} C_\gamma \frac{j!}{\gamma^j}\|v\|_2^{2j-4} \rt)\\
&\leq \|v\|_2^2 \|\na v\|_2^2\lt(-\beta +C_\gamma \sum_{j\geq 2} \frac{(2\beta)^j}{\gamma^j} j\rt)\\
&\leq \beta \|v\|_2^2 \|\na v\|_2^2\lt(-1 + \beta\frac{4C_\gamma}{\gamma^2} \sum_{j\geq 0} (j+ 2)\frac{(2\beta)^j}{\gamma^j}\rt)
\end{align*}
for any $0 < \gamma < 4\pi$. Choose $\gamma =3\pi$, for $\beta < \pi$, we then have
\begin{align*}
\frac{d}{dt} J[w_t] \bigg|_{t=1} & \leq \beta \|v\|_2^2 \|\na v\|_2^2\lt(-1 + \beta\frac{4C_{3\pi}}{(3\pi)^2} \sum_{j\geq 0} (j+ 2)\frac{2^j}{3^j}\rt) = \beta \|v\|_2^2 \|\na v\|_2^2\lt(-1 + C \beta\rt),
\end{align*}
with $C =\frac{4C_{3\pi}}{(3\pi)^2} \sum_{j\geq 0} (j+ 2)\frac{2^j}{3^j}$. Hence,
\[
\frac{d}{dt} J[w_t] \bigg|_{t=1} <0
\]
for any $\beta < \max\{\pi, 1/C\}$. This contradicts with \eqref{eq:critical}, hence there is no $v$ which is critical point of $J$ in $\mathcal M$. This completes our proof.
\end{proof}

\section{The critical case}
For $\ep >0$ small enough, we have known that $MT(N,\beta_N-\ep,\al)$ is attained by a function $u_\ep \in H^{1,N}(\R^N)$ with $\|u_\ep\|_{H^{1,N}(\R^N)} =1$ and
\[
MT(N,\beta_N -\ep,\alpha) = \int_{\R^N} \Phi_N((\beta_N -\ep)(1+\al \|u_\ep\|_N^N) |u_\ep|^{\frac N{N-1}}) dx.
\]
We can assume that $u_\ep$ is decreasing radially symmetric, nonnegative. We use again the notation $u_\ep(r)$ instead of $u_\ep(x)$ with $r =|x|$. A straightforward compuation shows that the Euler--Lagrange equation of $u_\ep$ is given by
\begin{equation}\label{eq:EL}
\begin{cases}
-\De_N u_\ep + u_\ep^{N-1} = \frac{\al_\ep}{\lam_\ep} u_\ep^{\frac1{N-1}} \Phi_N'(\beta_\ep(1+\al \|u_\ep\|_N^N)^{\frac1{N-1}} u_\ep^{\frac N{N-1}}) + \gamma_\ep u_\ep^{N-1}&\mbox{in $\R^N$,}\\
u_\ep \in W^{1,N}(\R^N),\quad u_\ep >0,\quad \|\na u_\ep\|_N^N + \|u_\ep\|_N^N =1,\\
\beta_\ep = \beta_N -\ep, \quad \alpha_\ep =\frac{1+ \al \|u_\ep\|_N^N}{1+2\al\|u_\ep\|_N^N},\quad \gamma_\ep =\frac{\al}{1+2\al\|u_\ep\|_N^N},\\
\lam_\ep = \int_{\R^N} \Phi_N'(\beta_\ep(1+\al \|u_\ep\|_N^N)^{\frac1{N-1}} u_\ep^{\frac N{N-1}}) u_\ep^{\frac{N}{N-1}} dx,
\end{cases}
\end{equation}
here $-\De_N u_\ep = \text{\rm div}(|\na u_\ep|^{N-2} \na u_\ep)$ denotes the $N-$Laplace of $u_\ep$. Applying the standard elliptic estimates (see, e.g., \cite{Serrin64,Tolksdorf}) to \eqref{eq:EL} we have $u_\ep \in C^1(\R^N)$.

First we prove that $\lam_\ep$ is bounded away from zero. More precisely, we will prove the following.

\begin{lemma}\label{awayzero}
Let $\lam_\ep$ be defined as in \eqref{eq:EL}. Then it holds $\liminf_{\ep \to 0} \lam_\ep >0$.
\end{lemma}
\begin{proof}
We first claim that
\begin{equation}\label{eq:cont}
\lim_{\ep\to 0} MT(N,\beta_\ep,\al) = MT(N,\beta_N,\ep).
\end{equation}
Indeed, it is evident that $MT(N,\beta_\ep,\al) \leq MT(N,\beta_N,\al)$, then
\begin{equation}\label{eq:limsupMT}
\limsup_{\ep\to 0} MT(N,\beta_\ep,\al) \leq MT(N,\beta_N,\al).
\end{equation}
In the other hand, for any $u\in W^{1,N}(\R^N)$ with $\|u\|_{W^{1,N}(\R^N)}\leq 1$, applying Fatou's lemma we get
\[
\liminf_{\ep\to 0} \int_{\R^N} \Phi_N(\beta_\ep (1+\al\|u\|_N^N)^{\frac1{N-1}} |u|^{\frac N{N-1}}) dx \geq \int_{\R^N} \Phi_N(\beta_N (1+\al\|u\|_N^N)^{\frac1{N-1}} |u|^{\frac N{N-1}}) dx,
\]
which implies
\[
\liminf_{\ep\to 0} MT(N,\beta_\ep,\al) \geq \int_{\R^N} \Phi_N(\beta_N (1+\al\|u\|_N^N)^{\frac1{N-1}} |u|^{\frac N{N-1}}) dx.
\]
Taking the supremum over $u\in W^{1,N}(\R^N)$ with $\|u\|_{W^{1,N}(\R^N)}\leq 1$, we get
\begin{equation}\label{eq:liminfMT}
\liminf_{\ep\to 0} MT(N,\beta_\ep,\al) \geq MT(N,\beta_N,\al).
\end{equation}
Combining \eqref{eq:limsupMT} and \eqref{eq:liminfMT}, we obtain our claim \eqref{eq:cont}.

Note that
\[
t \Phi_N'(t) = \sum_{k=N-1}^{\infty} \frac{t^k}{(k-1)!} \geq \sum_{k=N-1}^\infty \frac{t^k}{k!} = \Phi_N(t), \quad t \geq 0,
\]
which then implies
\begin{align*}
\lam_\ep &= \int_{\R^N} \Phi_N'(\beta_\ep(1+\al \|u_\ep\|_N^N)^{\frac1{N-1}} u_\ep^{\frac N{N-1}}) u_\ep^{\frac{N}{N-1}} dx\\
&\geq \frac{1}{\beta_\ep (1+ \al \|u_\ep\|_N^N)^{\frac1{N-1}}} MT(N,\beta_\ep,\al)\\
&\geq \frac{1}{\beta_\ep (1+ \al)^{\frac1{N-1}}} MT(N,\beta_\ep,\al),
\end{align*}
here we use $\|u_\ep\|_N^N \geq 1$. Letting $\ep \to 0$ and using our claim \eqref{eq:cont}, we obtain
\[
\liminf_{\ep\to 0} \geq \frac{1}{\beta_N (1+\al)^{\frac1{N-1}}} MT(N,\beta_N,\al) >0,
\]
as wanted.
\end{proof}

Since $\{u_\ep\}_\ep$ is bounded in $W^{1,N}(\R^N)$, we can assume that $u_\ep \rightharpoonup  u_0$ weakly in $W^{1,N}(\R^N)$, $u_\ep \to u_0$ in $L^p_{\rm loc}(\R^N)$ for any $p < \infty$, $u_\ep \to u_0$ a.e. in $\R^N$, and $\|u_\ep\|_N^N \to a \in [0,1]$. In the sequel, we do not distinguish the sequence and subsequence, the interest reader should understand it from the context. Let 
\[
c_\ep = u_\ep(0) =\max_{x\in \R^N} u_\ep(x). 
\]
We have the following result.

\begin{lemma}\label{boundedcase}
If $c_\ep$ is bounded, then $MT(N,\beta_N,\al)$ is attained.
\end{lemma}
\begin{proof}
Since $c_\ep$ is bounded, we have $u_\ep \to u_0$ in $C^1_{\rm loc}(\R^N)$ by applying standard elliptic estimates to \eqref{eq:EL}. For any $R >0$ using Lemma \ref{radiallemma} and the argument in proof of \eqref{eq:outsideBR}, we get
\begin{equation}\label{eq:outsideBR*}
\lim_{R\to\infty}\lim_{\ep\to 0} \int_{B_R^c} \lt(\Phi_N(\beta_\ep (1+\al\|u_\ep\|_N^N)^{\frac1{N-1}} u_\ep^{\frac N{N-1}}) -\frac{\beta_\ep^{N-1}}{(N-1)!} (1+\al \|u_\ep\|_N^N) u_\ep^N \rt) dx =0.
\end{equation}
Since $u_\ep \to u_0$ in $C^1(\o{B_R})$, then it holds
\begin{align}\label{eq:insideBR*}
\lim_{R\to\infty} \lim_{\ep\to 0} \int_{B_R} &\lt(\Phi_N(\beta_\ep (1+\al\|u_\ep\|_N^N)^{\frac1{N-1}} u_\ep^{\frac N{N-1}}) -\frac{\beta_\ep^{N-1}}{(N-1)!} (1+\al \|u_\ep\|_N^N) u_\ep^N \rt) dx\notag\\
&=\lim_{R\to\infty} \int_{B_R} \lt(\Phi_N(\beta_N (1+\al a)^{\frac1{N-1}} u_0^{\frac N{N-1}}) -\frac{\beta_N^{N-1}}{(N-1)!} (1+\al a) u_0^N \rt) dx\notag\\
&=\int_{\R^N} \lt(\Phi_N(\beta_N (1+\al a)^{\frac1{N-1}} u_0^{\frac N{N-1}}) -\frac{\beta_N^{N-1}}{(N-1)!} (1+\al a) u_0^N \rt) dx
\end{align}
Combining \eqref{eq:outsideBR*} and \eqref{eq:insideBR*}, we have
\begin{align}\label{eq:hihi}
MT(N,\beta_N,\al) &=\int_{\R^N} \Phi_N(\beta_N (1+\al a)^{\frac1{N-1}} u_0^{\frac N{N-1}}) dx + \frac{\beta_N^{N-1}}{(N-1)!} (1+\al a) (a -\|u_0\|_N^N)
\end{align}
If $u_0 \equiv 0$, then
\[
MT(N,\beta_N,\alpha) = \frac{\beta_N^{N-1}}{(N-1)!} (1+\al a) a \leq \frac{\beta_N^{N-1}}{(N-1)!} (1+\al),
\]
which contradicts with Proposition \ref{lowerboundsub}. Hence $u_0\not\equiv 0$. Denote $\tau^N = a/\|u_0\|_N^N \geq 1$, and define the new function $v_0(x) = u_0(x/\tau)$ then $\|v_0\|_N^N = \tau^N \|u_0\|_N^N =a$ and using the lower semi-continuity of the Sobolev norm under the weak convergence, we have $\|v_0\|_{W^{1,N}(\R^N)} \leq 1$. By definition of $MT(N,\beta_N,\al)$, we then have
\begin{align*}
MT(N,\beta_N,\al) &\geq \int_{\R^N} \Phi_N(\beta_N (1 +\al \|v_0\|_N^N )^{\frac1{N-1}} v_0^{\frac{N}{N-1}}) dx\\
&=\tau^N \int_{\R^N} \Phi_N(\beta_N (1+\al a)^{\frac1{N-1}} u_0^{\frac N{N-1}}) dx\\
&=\int_{\R^N} \Phi_N(\beta_N (1+\al a)^{\frac1{N-1}} u_0^{\frac N{N-1}}) dx + (\tau^N-1)\frac{\beta_N^{N-1}}{(N-1)!}(1+ \al a) \|u_0\|_N^N\\
&\quad +(\tau^N-1) \int_{\R^N}\Bigg(\Phi_N(\beta_N (1 +\al a)^{\frac1{N-1}} u_0^{\frac N{N-1}}) - \frac{\beta_N^{N-1}}{(N-1)!} (1+ \al a) u_0^{N}\Bigg) dx\\
&=\int_{\R^N} \Phi_N(\beta_N (1+\al a)^{\frac1{N-1}} u_0^{\frac N{N-1}}) dx + \frac{\beta_N^{N-1}}{(N-1)!}(1+ \al a)(a- \|u_0\|_N^N)\\
&\quad +(\tau^N-1) \int_{\R^N}\Bigg(\Phi_N(\beta_N (1 +\al a)^{\frac1{N-1}} u_0^{\frac N{N-1}}) - \frac{\beta_N^{N-1}}{(N-1)!} (1+ \al a) u_0^{N}\Bigg) dx\\
&=MT(N,\beta_N,\al) + (\tau^N-1) \int_{\R^N}\Bigg(\Phi_N(\beta_N (1 +\al a)^{\frac1{N-1}} u_0^{\frac N{N-1}})\\
&\hspace{8cm} - \frac{\beta_N^{N-1}}{(N-1)!} (1+ \al a) u_0^{N}\Bigg) dx,
\end{align*}
here we use \eqref{eq:hihi}. Since $u_0\not\equiv 0$, hence $\tau =1$ or equivalently $a =\|u_0\|_N^N$. Using again \eqref{eq:hihi}, we get 
\begin{equation}\label{eq:hihi*}
MT(N,\beta_N,\al) =\int_{\R^N} \Phi_N(\beta_N (1+\al \|u_0\|_N^N)^{\frac1{N-1}} u_0^{\frac N{N-1}}) dx.
\end{equation}
Note that $\|u_0\|_{W^{1,N}(\R^N)} \leq 1$ by the lower semi-continuity of the Sobolev norm under the weak convergence. This fact together with \eqref{eq:hihi*} implies $\|u_0\|_{W^{1,N}(\R^N)} =1$ and hence $u_0$ is a maximizer for $MT(N,\beta_N,\al)$. This finishes our proof.
\end{proof}

In the sequel, we assume that $c_\ep \to \infty$. Under this assumption, we have the following result.

\begin{lemma}\label{azero}
It holds $|\na u_\ep|^N dx \rightharpoonup  \de_0$ weakly in the sense of measure. Consequently, we have $u_0 \equiv 0$, $\al_\ep \to 1$ and $\gamma_\ep \to \al$ as $\ep \to 0$.
\end{lemma}

\begin{proof}
Indeed, if  $|\na u_\ep|^N dx \not\rightharpoonup \de_0$ weakly in the sense of measure, then there exists $R >0$ and $\mu < 1$ such that
\[
\lim_{\ep \to 0} \int_{B_R} |\na u_\ep|^N dx < \mu < 1.
\]
Define $v_{\ep,R}(r) = u_\ep(r) -u_{\ep}(R)$ for $r < R$. We have $v_{\ep,R} \in W^{1,N}_0(B_R)$,
\[
\int_{B_R}|\na v_{\ep,R}|^N dx =\int_{B_R}|\na u_\ep|^N dx
\]
and by Lemma \ref{radiallemma}
\begin{align*}
u_\ep(r)^{\frac N{N-1}} &\leq (1+\de) v_{\ep,R}(r)^{\frac N{N-1}} + (1-(1+\de)^{\frac1{1-N}})^{1-N} u_\ep(R)^{\frac N{N-1}}\\
&\leq (1+\de) v_{\ep,R}(r)^{\frac N{N-1}} + (1-(1+\de)^{\frac1{1-N}})^{1-N} \frac{C}R
\end{align*}
for any $\de >0$.

Define $w_{\ep,R} = v_{\ep,R}/\|\na v_{\ep,R}\|_N$. On $B_R$, we have
\begin{align*}
\Phi_N'(\beta_\ep (1+\al \|u_\ep\|_N^N)^{\frac1{N-1}} u_\ep^{\frac N{N-1}}) &\leq e^{\beta_\ep (1+\al \|u_\ep\|_N^N)^{\frac1{N-1}} u_\ep^{\frac N{N-1}}}\\
&\leq e^{\beta_N(1+\al)(1 -(1+\de)^{\frac1{1-N}})^{1-N} \frac C{R}} e^{\beta_\ep(1+\de) (1+\al \|u_\ep\|_N^N)^{\frac1{N-1}} v_{\ep,R}^{\frac N{N-1}}}\\
&\leq C_N(R,\de) e^{\beta_\ep(1+\de)[(1+\al\|u_\ep\|_N^N) \|\na v_{\ep,R}\|_N^N]^{\frac1{N-1}}  w_{\ep,R}^{\frac N{N-1}}},
\end{align*}
with
\[
C_N(R,\de) =e^{\beta_N(1+\al)(1 -(1+\de)^{\frac1{1-N}})^{1-N} \frac C{R}}.
\]
Note that
\[
(1+\al\|u_\ep\|_N^N) \|\na v_{\ep,R}\|_N^N = (1+ \al -\al\|\na u_\ep\|_N^N)\|\na v_{\ep,R}\|_N^N \leq (1+ \al -\al\|\na v_{\ep,R}\|_N^N)\|\na v_{\ep,R}\|_N^N.
\]
We then have
\[
\lim_{\ep \to 0} \beta_\ep (1+\de)(1+\al\|u_\ep\|_N^N) \|\na v_{\ep,R}\|_N^N \leq \beta_N (1+\de)(1+ \al -\al \mu) \mu < \beta_N.
\]
for $\de >0$ small enough since $\al, \mu < 1$. Applying the classical Moser--Trudinger inequality \eqref{eq:usualMT}, we get that $e^{\beta_\ep(1+\de)[(1+\al\|u_\ep\|_N^N) \|\na v_{\ep,R}\|_N^N]^{\frac1{N-1}}  w_{\ep,R}^{\frac N{N-1}}}$ is bounded in $L^q(B_R)$ for some $q >1$. Therefore, $\Phi_N'(\beta_\ep (1+\al \|u_\ep\|_N^N)^{\frac1{N-1}} u_\ep^{\frac N{N-1}})$ is bounded in $L^q(B_R)$ for some $q > 1$. Since $u_\ep$ is bounded in $L^p(B_R)$ for any $p < \infty$. This together H\"older inequality shows that the function
\[
f_\ep = \frac{\al_\ep}{\lam_\ep} u_\ep^{\frac1{N-1}} \Phi_N'(\beta_\ep(1+\al \|u_\ep\|_N^N)^{\frac1{N-1}} u_\ep^{\frac N{N-1}}) + (\gamma_\ep-1) u_\ep^{N-1}
\]
is bounded in $L^s(B_R)$ for some $s >1$, here we use Lemma \ref{awayzero}. By the standard elliptic estimates to \eqref{eq:EL}, $u_\ep$ is uniformly bounded in $B_{R/2}$ which contradicts to $c_\ep \to \infty$. Then $|\na u_\ep|^N dx \rightharpoonup  \de_0$ weakly in the sense of measure.

Since $\|\na u_\ep\|_N^N + \|u_\ep\|_N^N =1$, we then have $\|u_\ep\|_N^N \to 0$ as $\ep \to 0$. Hence $u_\ep \to 0$ in $L^N(\R^N)$ which forces $u_0\equiv 0$. Consequently, $\al_\ep \to 1$ and $\gamma_\ep \to \al$ as $\ep \to 0$.
\end{proof}

Define 
\[
r_\ep^N = \frac{\lam_\ep}{\al_\ep} c_\ep^{-\frac N{N-1}} e^{-\beta_\ep(1+\al \|u_\ep\|_N^N)^{\frac1{N-1}} c_\ep^{\frac N{N-1}}}.
\]
We claim that

\begin{lemma}\label{eq:reptoinfinite}
It holds $\lim_{\ep \to 0} r_\ep^N =0.$
\end{lemma}
\begin{proof}
For any $0< \gamma < \beta_N$ and $R >0$, we have
\begin{align}\label{eq:splitrep}
r_\ep^N c_\ep^{\frac N{N-1}}e^{\gamma c_\ep^{\frac{N}{N-1}}}&= \frac{e^{(\gamma-\beta_\ep(1+\al\|u_\ep\|_N^N)^{\frac1{N-1}})c_\ep^{\frac N{N-1}}}}{\al_\ep }\int_{B_R}u_\ep^{\frac N{N-1}} \Phi_N'(\beta_\ep(1+\al\|u_\ep\|_N^N)^{\frac1{N-1}} u_\ep^{\frac N{N-1}}) dx\notag\\
&\quad + \frac{e^{(\gamma-\beta_\ep(1+\al\|u_\ep\|_N^N)^{\frac1{N-1}})c_\ep^{\frac N{N-1}}}}{\al_\ep}\int_{B_R^c}u_\ep^{\frac N{N-1}} \Phi_N'(\beta_\ep(1+\al\|u_\ep\|_N^N)^{\frac1{N-1}} u_\ep^{\frac N{N-1}}) dx\notag\\
&= I + II.
\end{align}
We first estimate $I$. Since $\|u_\ep\|_N \to 0$ and $\beta_\ep \to \beta_N$, then for $\ep >0$ small enough, we have
\begin{align*}
I&\leq \frac{e^{(\gamma-\beta_\ep(1+\al\|u_\ep\|_N^N)^{\frac1{N-1}})c_\ep^{\frac N{N-1}}}}{\al_\ep}\int_{B_R}u_\ep^{\frac N{N-1}} e^{\beta_\ep(1+\al\|u_\ep\|_N^N)^{\frac1{N-1}} u_\ep^{\frac N{N-1}}} dx\\
&\leq \frac1{\al_\ep}\int_{B_R}u_\ep^{\frac N{N-1}} e^{\gamma u_\ep^{\frac N{N-1}}} dx.
\end{align*}
Denote $v_\ep = u_\ep - u_\ep(R)$. We have $v_\ep \in W^{1,N}_0(B_R)$, $\|\na v_\ep\|_N \leq 1$ and by Lemma \ref{radiallemma}, there exists $C$ depends only on $N$ such that
\[
u_\ep(r)^{\frac N{N-1}} \leq (1+ \de) v_\ep^{\frac N{N-1}} + (1-(1+\de)^{\frac1{1-N}})^{1-N} \frac C R,
\]
for any $\de >0$. Therefore
\[
I \leq \frac{C_N(\de,R)}{\al_\ep} \int_{B_R}u_\ep^{\frac N{N-1}} e^{\gamma(1+ \de) v_\ep^{\frac N{N-1}}} dx,
\]
with
\[
C_N(\de,R) = e^{\gamma(1-(1+\de)^{\frac1{1-N}})^{1-N} \frac C R}.
\]
Choosing $\de >0$ small enough such that $\de \gamma < \beta_N$, and then applying H\"older inequality, the classical Moser--Trudinger inequality and the facts $u_\ep \to 0$ in $L^p_{\rm loc}(\R^N)$ for any $p < \infty$ and $\al_\ep \to 1$, we see that
\begin{equation}\label{eq:trongBR}
I = o_\ep(R),
\end{equation}
here $o_{\ep}(R)$ means that $\lim_{\ep\to 0} o_\ep(R) =0$ for $R$ is fixed.

For $II$, we note that
\[
t \Phi_N'(t) = \sum_{k=N-1}^\infty \frac{t^k}{(k-1)!}.
\]
Hence, for any $A >0$, there exists $C(A)$ such that $t\Phi_N'(t) \leq C(A) t^k$ for any $t \leq A$. Applying Lemma \ref{radiallemma}, there exists $C$ depending only on $N$ such that 
\[
u_\ep^{\frac{N}{N-1}}(r) \leq \frac C R,
\]
for any $r \geq $. Since $\beta_\ep(1+\al\|u_\ep\|_N^N)^{\frac1{N-1}} \leq 2^{\frac1{N-1}} \beta_N$, then there exists $C_N(R)$ depending only on $N$ and $R$ such that
\[
u_\ep^{\frac N{N-1}} \Phi_N'(\beta_\ep(1+\al\|u_\ep\|_N^N)^{\frac1{N-1}} u_\ep^{\frac N{N-1}}) dx \leq C_N(R) u_\ep^N,\quad\text{\rm on }\quad B_R^c.
\]
Thus, we get
\begin{equation}\label{eq:ngoaiBR}
II \leq C_N(R)\frac{e^{(\gamma -\beta_\ep(1+\al \|u_\ep\|_N^N)^{\frac1{N-1}})c_\ep^{\frac{N}{N-1}}}}{\al_\ep} =o_\ep(R).
\end{equation}
Combining \eqref{eq:trongBR} and \eqref{eq:ngoaiBR}, we get
\begin{equation}\label{eq:conclude}
\lim_{\ep\to 0} r_\ep^N c_\ep^{\frac N{N-1}}e^{\gamma c_\ep^{\frac{N}{N-1}}} =0,
\end{equation}
for any $0< \gamma < \beta_N$ which proves this lemma.\\
\end{proof}

We next define two new sequences of functions on $\R^N$ by
\begin{equation}\label{eq:blowupsequences}
\psi_\ep(x) = \frac{u_\ep(r_\ep x)}{c_\ep},\quad \vphi_\ep(x) = c_\ep^{\frac1{N-1}}(u_\ep(r_\ep x) -c_\ep),\quad x\in \R^N.
\end{equation}
From \eqref{eq:EL} and definition of $r_\ep$, we see that $\psi_\ep$ satisfies
\begin{align}\label{eq:psiepequation}
-\De_N \psi_\ep &= \frac1{c_\ep^N} \psi_\ep^{\frac1{N-1}} e^{\beta_\ep c_\ep^{\frac N{N-1}}(1+ \al \|u_\ep\|_N^N)^{\frac1{N-1}} (\psi_\ep^{\frac N{N-1}} -1)} + (\gamma_\ep -1)r_\ep^N \psi_\ep^{N-1} \notag\\
&\quad -\frac{e^{-\beta_\ep(1+\al\|u_\ep\|_N^N)^{\frac1{N-1}} c_\ep^{\frac N{N-1}}}}{c_\ep^N} \psi_\ep^{\frac1{N-1}} \sum_{k=0}^{N-3}\frac{ \beta_\ep^k (1+\al\|u_\ep\|_N^N)^{\frac k{N-1}}}{k!} c_\ep^{\frac{kN}{N-1}} \psi_\ep^{\frac{kN}{N-1}},
\end{align}
with remark that the last term does not appear if $N=2$. Since $\psi_\ep \leq 1$, applying the standard elliptic estimates to \eqref{eq:psiepequation}, we get that $\psi_\ep \to \psi$ in $C^1_{\rm loc}(\R^N)$ with $\psi$ satisfies $-\De_N \psi =0$. Obviously, $0\leq \psi \leq 1$, applying Liouville-type theorem for $N-$harmonic function, we conclude that $\psi\equiv 1$. Thus we have proved
\begin{lemma}\label{limpsiep}
It holds $\psi_\ep \to 1$ in $C^1_{\rm loc}(\R^N)$.
\end{lemma}

Similarly, $\vphi_\ep$ satisfies
\begin{align}\label{eq:vphiepequation}
-\De_N \vphi_\ep &= \psi_\ep^{\frac1{N-1}} e^{\beta_\ep c_\ep^{\frac N{N-1}}(1+ \al \|u_\ep\|_N^N)^{\frac1{N-1}} (\psi_\ep^{\frac N{N-1}} -1)} + (\gamma_\ep -1)r_\ep^N c_\ep u_\ep^{N-1} \notag\\
&\quad -e^{-\beta_\ep(1+\al\|u_\ep\|_N^N)^{\frac1{N-1}} c_\ep^{\frac N{N-1}}}\psi_\ep^{\frac1{N-1}} \sum_{k=0}^{N-3}\frac{ \beta_\ep^k (1+\al\|u_\ep\|_N^N)^{\frac k{N-1}}}{k!} u_\ep^{\frac{kN}{N-1}},
\end{align}
with the remark that the last term does not appear if $N=2$. Recall that $\psi_\ep \leq 1$. Applying \eqref{eq:conclude} and the standard elliptic estimates to \eqref{eq:vphiepequation}, we get that $\vphi_\ep \to \vphi$ in $C^1_{\rm loc}(\R^N)$ for some function $\vphi\in C^1(\R^N)$. Fix $R >0$, we have known that 
\[
\sup_{x\in B_R} |\psi_\ep(x) -1| = o_\ep(R).
\]
We also have
\begin{align*}
c_\ep^{\frac N{N-1}}(\psi_\ep(x)^{\frac{N}{N-1}} -1) = \frac N{N-1} \vphi_\ep  + O(|\vphi_\ep(x)| |\psi_\ep -1|)
\end{align*}
for $x\in B_R$, hence
\begin{equation}\label{eq:uniformvphi}
c_\ep^{\frac N{N-1}}(\psi_\ep(x)^{\frac{N}{N-1}} -1) \to \frac{N}{N-1} \vphi,
\end{equation}
uniformly on $B_R$. Thus, we obtain from \eqref{eq:vphiepequation} that
\begin{equation}
-\De_N \vphi = e^{-\frac{N}{N-1} \beta_N \vphi}.
\end{equation}
In the other hand, for any $R > 0$, we have $u_\ep(x) = c_\ep(1+o_\ep(R))$ uniformly on $B_{Rr_\ep}$ by Lemma \ref{limpsiep}. Thus, for any $p$, we have
\begin{equation}\label{eq:vocungnho1}
\frac{\int_{B_{Rr_\ep}} u_\ep^p dx}{\int_{B_{Rr_\ep}} e^{\beta_\ep (1+\al \|u_\ep\|_N^N)^{\frac1{N-1}} u_\ep^{\frac N{N-1}}} dx} = o_\ep(R),
\end{equation}
here we use again \eqref{eq:conclude}. Therefore
\begin{align*}
\int_{B_R} e^{-\frac N{N-1} \beta_N \vphi} dx &= \lim_{\ep\to 0} \int_{B_R} e^{\beta_\ep (1+ \al\|u_\ep\|_N^N)^{\frac1{N-1}}( u_\ep(r_\ep x)^{\frac N{N-1}} -c_\ep^{\frac N{N-1}})} dx\\
&=\lim_{\ep\to 0} \al_\ep c_\ep^{\frac N{N-1}}\frac{\int_{B_{Rr_\ep}} e^{\beta_\ep(1+ \alpha\|u_\ep\|_N^N)^{\frac1{N-1}} u_\ep^{\frac N{N-1}}} dx}{\int_{\R^N} u_\ep^{\frac N{N-1}} \Phi_N'(\beta_\ep (1+\al \|u_\ep\|_N^N)^{\frac1{N-1}} u_\ep^{\frac N{N-1}})dx}\\
&\leq \lim_{\ep\to 0} \al_\ep\frac{c_\ep^{\frac N{N-1}}\int_{B_{Rr_\ep}} e^{\beta_\ep(1+ \alpha\|u_\ep\|_N^N)^{\frac1{N-1}} u_\ep^{\frac N{N-1}}} dx}{\int_{B_{Rr_\ep}} u_\ep^{\frac N{N-1}} \Phi_N'(\beta_\ep (1+\al \|u_\ep\|_N^N)^{\frac1{N-1}} u_\ep^{\frac N{N-1}})dx}\\
&=\lim_{\ep\to 0}\al_\ep \frac{\int_{B_{Rr_\ep}} e^{\beta_\ep(1+ \alpha\|u_\ep\|_N^N)^{\frac1{N-1}} u_\ep^{\frac N{N-1}}} dx}{\int_{B_{Rr_\ep}}  \Phi_N'(\beta_\ep (1+\al \|u_\ep\|_N^N)^{\frac1{N-1}} u_\ep^{\frac N{N-1}}) dx (1+ o_\ep(R))}\\
&=\lim_{\ep\to 0}\al_\ep \frac{\int_{B_{Rr_\ep}} e^{\beta_\ep(1+ \alpha\|u_\ep\|_N^N)^{\frac1{N-1}} u_\ep^{\frac N{N-1}}} dx}{\int_{B_{Rr_\ep}} e^{\beta_\ep (1+\al \|u_\ep\|_N^N)^{\frac1{N-1}} u_\ep^{\frac N{N-1}}} dx (1+ o_\ep(R))}\\
&=1,
\end{align*}
here we use \eqref{eq:vocungnho1} and the estimate $u_\ep =c_\ep(1+ o_\ep(R))$ uniformly on $B_{Rr_\ep}$. Let $R\to \infty$, we get
\begin{equation}\label{eq:tichphanhuuhan}
\int_{\R^N} e^{\frac N{N-1} \beta_N \vphi} dx \leq 1.
\end{equation}
It is obvious that $\vphi(x) \leq \vphi(0) =0$. Repeating the argument in \cite{Yang06*} or using a recent classification result for the quasi-linear Liouville equation of Esposito \cite{Esposito} , we conclude that
\begin{equation}\label{eq:vphian}
\vphi(x) = -\frac{N-1}{\beta_N} \ln \lt(1 + \lt(\frac{\om_{N-1}}{N}\rt)^{\frac1{N-1}} |x|^{\frac N{N-1}}\rt),
\end{equation}
and
\begin{equation}\label{eq:tichphan1}
\int_{\R^N} e^{\frac N{N-1} \beta_N \vphi} dx =1.
\end{equation}
Remark that when $N=2$, the representation of $\vphi$ is given by the classification result of Chen and Li \cite{ChenLi}. Thus, we have proved

\begin{lemma}\label{limvphiep}
It holds $\vphi_\ep \to \vphi$ in $C^1_{\rm loc} (\R^N)$, with $\vphi$ is given by \eqref{eq:vphian}.
\end{lemma}

We next consider the asymptotic behavior of $u_\ep$ away from zero. For $c >1$, let us denote $u_{\ep,c} = \min\{u_\ep, c_\ep/c\}$. We have the following result.
\begin{lemma}\label{truncationuep}
It holds $\lim_{\ep \to 0} \int_{\R^N} |\na u_{\ep,c}|^N dx  = \frac 1c$ for any $c >1$.
\end{lemma}
\begin{proof}
Using \eqref{eq:EL} and the fact $\|u_\ep\|_N^N =o_\ep(1)$, we have
\begin{align*}
\int_{\R^N} |\na u_{\ep,c}|^N dx & =\int_{\R^N}|\na u_\ep|^{N-2} \na u_\ep \cdot \na u_{\ep,c} dx\\
&=-\int_{\R^N} \Delta_N u_\ep \, u_{\ep,c} dx\\
&=o_\ep(1) + \frac{\al_\ep}{\lam_\ep} \int_{\R^N} u_\ep^{\frac1{N-1}} \Phi_N'(\beta_\ep(1+\al \|u_\ep\|_N^N)^{\frac1{N-1}} u_\ep^{\frac N{N-1}}) u_{\ep,c} dx.
\end{align*}
Fix $R >0$, since $u_\ep = c_\ep(1+ o_\ep(R))$ uniformly on $B_{Rr_\ep}$, then for $\ep>0$ small enough we have $B_{Rr_\ep} \subset \{u_\ep \geq c_\ep/c\}$. Hence
\begin{align*}
\int_{\R^N} |\na u_{\ep,c}|^N dx &\geq o_\ep(1) + \frac{\al_\ep}{\lam_\ep}\frac{c_\ep^{\frac{N}{N-1}}}c (1+o_\ep(R)) \int_{B_{Rr_\ep}} \Phi_N'(\beta_\ep(1+\al \|u_\ep\|_N^N)^{\frac1{N-1}} u_\ep^{\frac N{N-1}}) dx \\
&= o_\ep(1) + \frac{\al_\ep}{\lam_\ep}\frac{c_\ep^{\frac{N}{N-1}}}c (1+o_\ep(R)) \int_{B_{Rr_\ep}} e^{\beta_\ep(1+\al \|u_\ep\|_N^N)^{\frac1{N-1}} u_\ep^{\frac N{N-1}}} dx, 
\end{align*}
here we use \eqref{eq:vocungnho1}. Making the change of variable $x =r_\ep y$, and using \eqref{eq:uniformvphi}, we get
\[
\int_{\R^N} |\na u_{\ep,c}|^N dx \geq o_\ep(1) + \frac1c (1+o_\ep(R)) \int_{B_R} e^{(1+o_\ep(R))\frac N{N-1} \beta_N \vphi } dy.
\]
Let $\ep \to 0$, then let $R\to \infty$ and using \eqref{eq:tichphan1} we obtain
\begin{equation}\label{eq:abba1}
\liminf_{\ep\to 0} \int_{\R^N} |\na u_{\ep,c}|^N dx \geq \frac1c.
\end{equation}
Similarly, we have
\begin{equation}\label{eq:abba2}
\liminf_{\ep\to 0} \int_{\R^N} \lt|\na \lt(u_{\ep} -\frac{c_\ep}c\rt)_+\rt|^N dx \geq 1 -\frac1c.
\end{equation}
Since $\|u_\ep\|_N =o_\ep(1)$, then
\begin{equation}\label{eq:abba3}
\int_{\R^N} |\na u_{\ep,c}|^N dx + \int_{\R^N} \lt|\na \lt(u_{\ep} -\frac{c_\ep}c\rt)_+\rt|^N dx = \|\na u_\ep\|_N^N = 1 + o_\ep(1).
\end{equation}
Combining \eqref{eq:abba1}, \eqref{eq:abba2} and \eqref{eq:abba3}, we conclude the lemma.
\end{proof}

\begin{lemma}\label{uppersup}
We have
\begin{equation}\label{eq:uppersup}
MT(N,\beta_N,\al)= \lim_{\ep \to 0} \frac{\lam_\ep}{c_\ep^{\frac N{N-1}}}.
\end{equation}
As consequently, for any $\theta < N/(N-1)$, we have $\lam_\ep/c_\ep^{\theta} \to \infty$ as $\ep \to 0$. Furthermore, $c_\ep^{\frac N{N-1}}/\lam_\ep$ is bounded.
\end{lemma}
\begin{proof}
For any $c >1$, denote $u_\ep =\min\{u_\ep,c_\ep/c\}$. We have
\begin{align*}
MT(N,\beta_\ep,\al) & = \int_{\{u_\ep \leq c_\ep/c\}}\Phi_N(\beta_\ep(1+\al \|u_\ep\|_N^N) |u_\ep|^{\frac N{N-1}}) dx\\
&\quad  + \int_{\{u_\ep > c_\ep/c\}}\Phi_N(\beta_\ep (1+\al \|u_\ep\|_N^N) |u_\ep|^{\frac N{N-1}}) dx\\
&\leq \int_{\R^N}\Phi_N(\beta_\ep(1+\al \|u_\ep\|_N^N) |u_{\ep,c}|^{\frac N{N-1}}) dx  \\
&\quad + \frac{c^{\frac{N}{N-1}}}{c_\ep^{\frac N{N-1}}} \int_{\R^N} u_\ep^{\frac{N}{N-1}} \Phi_N(\beta_\ep (1+\al \|u_\ep\|_N^N) |u_\ep|^{\frac N{N-1}})dx \\
&\leq \int_{\R^N}\Phi_N(\beta_\ep(1+\al \|u_\ep\|_N^N) |u_{\ep,c}|^{\frac N{N-1}}) dx + c^{\frac N{N-1}} \frac{\lam_\ep}{c_\ep^{\frac N{N-1}}},
\end{align*}
here we use the inequality $\Phi_N(t) \leq \Phi_N'(t)$ for $t \geq 0$. By Lemma \ref{truncationuep}, we have
\begin{equation}\label{eq:ooo}
\lim_{\ep\to 0} (\|\na u_{\ep,c}\|_N^N + \|u_{\ep,c}\|_N^N) = \frac 1{c} < 1.
\end{equation}
Repeating the proof of \eqref{eq:trunlimit1} and \eqref{eq:trunlimit2} with the help of \eqref{eq:ooo}, the classical Moser--Trudinger inequality \eqref{eq:usualMT}, and the fact $\|u_\ep\|_N^N =o_\ep(1)$, we conclude that
\[
\lim_{\ep\to 0} \int_{\R^N}\Phi_N(\beta_\ep(1+\al \|u_\ep\|_N^N) |u_{\ep,c}|^{\frac N{N-1}}) dx = 0.
\]
Letting $\ep \to 0$ and then $c \downarrow 1$, we obtain
\begin{equation}\label{eq:a1}
MT(N,\beta_N,\al) = \liminf_{\ep\to 0} MT(N,\beta_\ep,N) \leq \liminf_{\ep\to 0} \frac{\lam_\ep}{c_\ep^{\frac N{N-1}}}.
\end{equation}
In the other hand, we have $\Phi_N'(t) =\Phi_N(t) + t^{N-2}/(N-2)!$, hence
\begin{align*}
\lam_\ep& =\int_{\R^N} u_\ep^{\frac N{N-1}}\Phi_N(\beta_\ep(1+\al \|u_\ep\|_N^N) |u_{\ep}|^{\frac N{N-1}}) dx + \frac{\beta_\ep^{N-2} (1+ \al\|u_\ep\|_N^N)^{\frac{N-2}{N-1}}}{(N-2)!} \|u_\ep\|_N^N\\
&\leq c_\ep^{\frac{N}{N-1}} MT(N,\beta_\ep,\al) +  \frac{\beta_\ep^{N-2} (1+ \al\|u_\ep\|_N^N)^{\frac{N-2}{N-1}}}{(N-2)!} \|u_\ep\|_N^N,
\end{align*}
which then implies
\begin{equation}\label{eq:a2}
\limsup_{\ep\to 0} \frac{\lam_\ep}{c_\ep^{\frac N{N-1}}} \leq MT(N,\beta_N,\al).
\end{equation}
Combining \eqref{eq:a1} and \eqref{eq:a2}, we get \eqref{eq:uppersup}.

The rest of this lemma is immediate consequence of \eqref{eq:uppersup}. Indeed, if otherwise, we then have $MT(N,\beta_N,\al) =0$ which is impossible.
\end{proof}

\begin{lemma}\label{deltafunction}
For any $\phi \in C_0^0(\R^N)$, we have
\begin{equation}\label{eq:delta0}
\lim_{\ep \to 0} \int_{\R^N} \frac{c_\ep \al_\ep}{\lam_\ep} u_\ep^{\frac1{N-1}} \Phi_N'(\beta_\ep (1+ \al\|u_\ep\|_N^N)^{\frac1{N-1}} u_\ep^{\frac N{N-1}}) \phi dx = \phi(0).
\end{equation}
\end{lemma}
\begin{proof}
Fix $c >1$ and $R >0$ , we divide $\R^N$ into three parts as follows
\[
\Om_1 =\{u_\ep \leq c_\ep/ c \},\quad \Om_2 = \{u_\ep > c_\ep/c\} \setminus B_{Rr_\ep},\quad \Om_3 = \{u_\ep > c_\ep/c\} \cap B_{Rr_\ep}.
\]
Note that $\Om_3 = B_{Rr_\ep}$ for $\ep >0$ small enough by Lemma \ref{limpsiep}. Denote 
\[
g_\ep = \frac{c_\ep \al_\ep}{\lam_\ep} u_\ep^{\frac1{N-1}} \Phi_N'(\beta_\ep (1+ \al\|u_\ep\|_N^N)^{\frac1{N-1}} u_\ep^{\frac N{N-1}}) \quad \text{and}\quad M =\max_{x\in \R^N} |\phi(x)|
\]
for simplifying notation. It is easy to see that $\int_{\R^N} u_\ep^{N-1} |\phi| dx = o_\ep(1)$. Thus we have on $\Om_1$ that
\begin{align*}
\lt|\int_{\Om_1} g_\ep \phi dx\rt|& \leq \frac{c_\ep \al_\ep}{\lam_\ep}M \int_{\R^N} u_{\ep,c}^{\frac1{N-1}} \Phi_N(\beta_\ep (1+ \al\|u_\ep\|_N^N)^{\frac1{N-1}} u_{\ep,c}^{\frac N{N-1}}) dx + o_\ep(1).
\end{align*}
Arguing as in the proof of \eqref{eq:trunlimit1} and \eqref{eq:trunlimit2} with the help of \eqref{eq:ooo}, the classical Moser--Trudinger inequality \eqref{eq:usualMT}, the fact $\|u_\ep\|_N^N =o_\ep(1)$ and Lemma \ref{uppersup}, we then get
\begin{equation}\label{eq:b1}
\int_{\Om_1} g_\ep \phi dx = o_\ep(1).
\end{equation}
We next claim that
\begin{equation}\label{eq:claim*}
\lim_{R\to\infty} \lim_{\ep\to 0} \frac{c_\ep}{\lam_\ep} \int_{B_{Rr_\ep}} u_\ep^{\frac1{N-1}} \Phi_N'(\beta_\ep (1+ \al\|u_\ep\|_N^N)^{\frac1{N-1}} u_\ep^{\frac N{N-1}}) dx =1.
\end{equation}
Indeed, by Lemma \ref{limpsiep}, we have $u_\ep =c_\ep(1+ o_\ep(R)$ uniformly on $B_{Rr_\ep}$. This together \eqref{eq:uniformvphi} and \eqref{eq:vocungnho1} gives
\begin{align*}
\frac{c_\ep}{\lam_\ep} \int_{B_{Rr_\ep}}u_\ep^{\frac1{N-1}} &\Phi_N'(\beta_\ep (1+ \al\|u_\ep\|_N^N)^{\frac1{N-1}} u_\ep^{\frac N{N-1}}) dx\\
& = (1+ o_\ep(R))\frac{c_\ep^{\frac N{N-1}}}{\lam_\ep} \int_{B_{Rr_\ep}} e^{\beta_\ep (1+ \al\|u_\ep\|_N^N)^{\frac1{N-1}} u_\ep^{\frac N{N-1}}} dx\\
&=(1+o_\ep(R)) \int_{B_R} e^{\beta_\ep(1+ \al\|u_\ep\|_N^N)^{\frac1{N-1}} c_\ep^{\frac N{N-1}} (\psi_\ep(x)^{\frac N{N-1}} -1)} dx\\
&=  (1+o_\ep(R)) \int_{B_R} e^{(1+o_\ep(R)) \beta_N \frac N{N-1} \vphi} dx.
\end{align*}
Let $\ep \to 0$ and then $R \to \infty$ we get our claim \eqref{eq:claim*}.

On $\Om_2$ we have
\begin{align}\label{eq:b2}
\lt|\int_{\Om_2} g_\ep \phi dx \rt| &\leq \frac{\al_\ep c_\ep}{\lam_\ep} M \int_{\Om_2} u_\ep^{\frac1{N-1}}\Phi_N'(\beta_\ep (1+ \al\|u_\ep\|_N^N)^{\frac1{N-1}} u_\ep^{\frac N{N-1}}) dx\notag\\
&\leq Mc \frac{\al_\ep}{\lam_\ep} \int_{\Om_2} u_\ep^{\frac N{N-1}}\Phi_N'(\beta_\ep (1+ \al\|u_\ep\|_N^N)^{\frac1{N-1}} u_\ep^{\frac N{N-1}}) dx\notag\\
&\leq Mc\al_\ep \lt(1 - \frac{1}{\lam_\ep} \int_{B_{Rr_\ep}} u_\ep^{\frac N{N-1}} \Phi_N'(\beta_\ep (1+ \al\|u_\ep\|_N^N)^{\frac1{N-1}} u_\ep^{\frac N{N-1}}) dx\rt)\notag\\
&= o_\ep(R) + o_R(1)
\end{align}
with $o_R(1) \to 0$ as $R \to \infty$, here we use \eqref{eq:claim*} and the fact $u_\ep =c_\ep(1+ o_\ep(R))$ uniformly on $B_{Rr_\ep}$.

On $\Om_3 =B_{Rr_\ep}$, we have by \eqref{eq:claim*} that
\[
\lt|\int_{\Om_3} g_\ep (\phi -\phi_0)\rt| \leq \sup_{x\in B_{Rr_\ep}}|\phi(x) -\phi(0)| \frac{\al_\ep c_\ep}{\lam_\ep} \int_{B_{Rr_\ep}} u_\ep^{\frac1{N-1}} \Phi_N'(\beta_\ep (1+ \al\|u_\ep\|_N^N)^{\frac1{N-1}} u_\ep^{\frac N{N-1}}) dx
\]
which then implies
\begin{equation}\label{eq:b3}
\int_{\Om_3} g_\ep \phi dx = \phi(0) + o_\ep(R) + o_R(1).
\end{equation}
Combining \eqref{eq:b1}, \eqref{eq:b2} and \eqref{eq:b3}, we obtain \eqref{eq:delta0}.
\end{proof}

\begin{lemma}\label{Greenfunction}
$c_\ep^{\frac1{N-1}} u_\ep \to G_\al$ in $C^1_{\rm loc}(\R^N)$ and weakly in $W_{\rm loc}^{1,q}(\R^N\setminus\{0\})$ for any $1< q < N$, where $G_\al$ is a distributional solution to
\begin{equation}\label{eq:Green}
-\De_N G_\al + G_\al^{N-1} = \de_0 + \al G_\al^{N-1}.
\end{equation}
Moreover $G_\al\in W^{1,N}(\R^N\setminus B_R)$ for any $R >0$ and takes the form 
\begin{equation}\label{eq:Greenform}
G_\al(x) = -\frac{N}{\beta_N} \ln |x| + A_\al + w(x),
\end{equation}
where $A_\al$ is constant, and $w\in C^1(\R^N)$ satisfies $w(x) = O(|x|^N \ln^{N-1}|x|)$ as $x\to 0$.
\end{lemma}
\begin{proof}
Denote $w_\ep = c_\ep^{\frac1{N-1}} u_\ep$. It follows from \eqref{eq:EL} that $w_\ep$ satisfies
\begin{equation}\label{eq:wep}
-\De_N w_\ep + w_\ep^{N-1} = \frac{\al_\ep c_\ep}{\lam_\ep} u_\ep^{\frac1{N-1}} \Phi_N'(\be_\ep (1+ \al \|u_\ep\|_N^N)^{\frac1{N-1}} u_\ep^{\frac N{N-1}}) + \gamma_\ep w_\ep^{N-1}.
\end{equation}
Denote $f_\ep = \frac{\al_\ep c_\ep}{\lam_\ep} u_\ep^{\frac1{N-1}} \Phi_N'(\be_\ep (1+ \al \|u_\ep\|_N^N)^{\frac1{N-1}} u_\ep^{\frac N{N-1}})$. By Lemma \ref{deltafunction}, the sequence $f_\ep$ is bounded in $L^1_{\rm loc}(\R^N)$. We first prove that the sequence $w_\ep$ is bounded in $W_{\rm loc}^{1,q}(\R^N)$ for any $1< q < N$. We will need the following result of Yang \cite{Yang06*,Yang07corrige}: \emph{Let $\Om$ be a smooth bounded domain in $\R^N$. Let $g_\ep$ is a bounded sequence in $L^1(\Om)$ and $v_\ep \in C^1(\o\Om) \cap W^{1,N}_0(\Om)$ satisfy
\[
-\De_N v_\ep = g_\ep,\quad\text{in}\quad \Om.
\]
Then for any $1< q < N$, there exists a constant $C$ depending only on $N,q$, $\Om$ and the upper bound of $\|g_\ep\|_{1,\Om}$ such that $\|\na v_\ep\|_{q,\Om} \leq C$.}

We claim that $w_\ep$ is bounded in $L^{N-1}_{\rm loc} (\R^N)$. Indeed, if otherwise, there exists $R >0$ such that $\|w_\ep\|_{N-1,B_R} \to \infty$ as $\ep \to 0$. Denote $v_\ep = w_\ep /\|w_\ep\|_{N-1,B_R}$, we have
\[
-\Delta_N v_\ep = \frac{f_\ep}{\|w_\ep\|_{N-1,B_R}^{N-1}} + (\gamma_\ep -1)v_\ep^{N-1} =: g_\ep.
\]
Note that $\|g_\ep\|_{1,B_R}$ is bounded. Let $\tilde v_\ep = v_\ep -v_\ep(R)$ then $\tilde v_\ep \in C^1(\o{B_R}) \cap W^{1,N}_0(B_R)$ and $-\De_N \tilde v_\ep = g_\ep$ in $B_R$. Applying the observation above, we get $\|\na v_\ep\|_{q,B_R} = \|\na \tilde v_\ep\|_{q,B_R}$ is bounded by a constant depending only on $q, R$ and $N$. Since $\|v_\ep\|_{N-1,B_R} =1$ for all $\ep$, then by Sobolev inequality, we get that $v_\ep$ is bounded in $W^{1,q}(B_R)$ for any $1< q < N$. Thus $v_\ep \to v_0$ weakly in $W^{1,q}(B_R)$ for any $1< q < N$ and in $L^{N-1}(B_R)$. Obviously, we have $\|v_0\|_{N-1,B_R} =1$ and 
\[
-\De_N v_0 = (\al-1) v_0\quad \text{ in } \quad B_R.
\]
Note that $v_0$ is a decreasing radially symmetric function on $B_R$, applying the standard elliptic regularity to the equation above, we have $v_0\in C^1(B_R)$ and is bounded. Taking $v_0$ as a test function, we get
\[
(\al -1) \|v_0\|_{N-1,B_R}^{N-1} = -\int_{B_R} \De_N v_0 \, v_0 dx = \int_{\pa B_R} |\na v_0|^{N-1} v_0 ds \geq 0,
\]
which forces $v_0\equiv 0$ since $\al < 1$, here $ds$ denotes the surface area measure on the sphere $\pa B_R$. This contradicts to $\|v_0\|_{N-1,B_R} =1$. Hence our claim is proved.

Back to \eqref{eq:wep}, we see that $-\De_N w_\ep$ is bounded in $L^1_{\rm loc}(\R^N)$. Repeating the argument above for $v_\ep$, we then get the boundedness of $\|\na w_\ep\|_{q,B_R}$ for any $1 < q < N$ and for any $R >0$. Since $w_\ep$ is bounded in $L^{N-1}_{\rm loc}(\R^N)$, by Sobolev inequality, we get that $w_\ep$ is bounded in $W_{\rm loc}^{1,q}(\R^N)$. Thus $w_\ep \rightharpoonup G_\al$ weakly in $W^{1,q}_{\rm loc}(\R^N)$ for any $1 < q < N$.

From Lemma \ref{uppersup} and radial lemma, we see that $f_\ep \to 0$ uniformly on $B_R^c$ for any $R >0$, hence $-\Delta_N w_\ep $ is bounded in $L^p_{\rm loc}(\R^N\setminus\{0\})$. Applying standard elliptic estimates to \eqref{eq:wep}, we get that $w_\ep \to G_\al$ in $C^1_{\rm loc}(\R^N\setminus\{0\})$.

Let $\ep \to 0$ and using Lemma \ref{deltafunction}, we easily obtain from \eqref{eq:wep} that $G_\al$ satisfies the equation \eqref{eq:Green}. Multiplying both side of \eqref{eq:wep} by $w_\ep$, integrating on $B_R^c$ and using the inequality $t\Phi_N'(t) \leq t^{N-1}e^t$, we get
\begin{align}\label{eq:xxxx}
&\int_{B_R^c}\lt(|\na w_\ep|^N + (1-\ga_\ep)w_\ep^N\rt) dx \notag\\
&= \int_{\pa B_R} |\na w_\ep|^{N-1} w_\ep ds + \frac{\al_\ep c_\ep^{\frac N{N-1}}}{\lam_\ep}\int_{B_R^c} u_\ep^{\frac N{N-1}} \Phi_N'(\be_\ep (1+ \al \|u_\ep\|_N^N)^{\frac1{N-1}} u_\ep^{\frac N{N-1}}) dx\notag\\
&\leq \int_{\pa B_R} |\na w_\ep|^{N-1} w_\ep ds + \frac{\al_\ep c_\ep^{\frac N{N-1}}}{\lam_\ep}\be_N^{N-2} (1+ \al)^{\frac{N-2}{N-1}}e^{\beta_N(1+\al)^{\frac1{N-1}}u_\ep(R)^{\frac N{N-1}}} \int_{B_R^c} u_\ep^{N} dx.
\end{align}
The last term on right hand side of \eqref{eq:xxxx} tends to zero because of Lemma \ref{uppersup} and the fact $\|u_\ep\|_N \to 0$. Using Fatou's lemma, we get
\[
\int_{B_R^c}\lt(|\na G_\al|^N + (1-\al)G_\al^{N}\rt) dx \leq \int_{\pa B_R} |\na G_\al|^{N-1} G_\al ds.
\]
Since $\al < 1$, we then have $G_\al \in W^{1,N}(\R^N \setminus B_R)$ for any $R > 0$. The form \eqref{eq:Greenform} of $G$ follows from the Lemma $3.8$ of Li and Ruf (see \cite{LiRuf2008}).
\end{proof}

We proceed by proving an upper bound for $MT(N,\beta_N,\al)$ under the assumptions that $c_\ep \to \infty$. More precisely, we have the following
\begin{lemma}\label{chantrenMT}
Under the assumption $c_\ep \to \infty$, we have
\begin{equation}\label{eq:chantrenMT}
MT(N,\beta_N,\al) \leq \frac{\om_{N-1}}N e^{\beta_N A_\al + 1 + \frac12 + \cdots + \frac1{N-1}},
\end{equation}
where $A_\al$ given in \eqref{eq:Greenform}.
\end{lemma}
\begin{proof}
For any $\de >0$, from the proof of Lemma \ref{Greenfunction}, we see that
\[
\frac{\al_\ep c_\ep^{\frac N{N-1}}}{\lam_\ep}\int_{B_R^c} u_\ep^{\frac N{N-1}} \Phi_N'(\be_\ep (1+ \al \|u_\ep\|_N^N)^{\frac1{N-1}} u_\ep^{\frac N{N-1}}) dx = o_\ep(\de),
\]
where $o_\ep(\de) \to 0$ as $\ep \to 0$ and $\de $ is fixed. Since $G_\al\in W^{1,N}_{\rm loc}(\R^N\setminus 0)$ then
\begin{equation}\label{eq:2x}
\int_{B_R^c} (|\na G_\al|^N + (1-\al) G_\al^N) dx = \int_{\pa B_R} |\na G_\al|^{N-1} G_\al ds.
\end{equation}
Hence
\[
\lim_{R\to \infty} \int_{\pa B_R} |\na G_\al|^{N-1} G_\al ds =0.
\]
This together \eqref{eq:xxxx} and \eqref{eq:2x} implies
\[
\lim_{R\to \infty}\lim_{\ep \to 0} \int_{B_R^c} |w_\ep|^N dx =0.
\]
Consequently, we have $w_\ep \to G_\al$ in $L^N(\R^N)$ since $w_\ep \to G_\al$ in $L^N_{\rm loc}(\R^N)$. Using \eqref{eq:xxxx} with $R$ replaced by $\de$, we have
\begin{align*}
\int_{B_\de^c} (|\na u_\ep|^N + u_\ep^N) dx &=\frac1{c_\ep^{\frac N{N-1}}}\lt(\int_{\pa B_\de} |\na G_\al|^{N-1} G_\al ds + \al \|G_\al\|_N^N + o_\ep(\de) + o_\de(1)\rt)\notag\\
&=\frac1{c_\ep^{\frac N{N-1}}}\lt(-\frac N{\beta_N} \ln \de + A_\al + \al \|G_\al\|_N^N + o_\ep(\de) + o_\de(1)\rt).
\end{align*}
Note that 
\[
\int_{B_\de} |u_\ep|^N dx =\frac1{c_\ep^{\frac N{N-1}}} \lt(\int_{B_\de} G_\al^N dx + o_\ep(\de)\rt) =\frac1{c_\ep^{\frac N{N-1}}} \lt(o_\de(1) + o_\ep(\de)\rt).
\]
Combining two previous estimates together the fact $\|\na u_\ep\|_N^N + \|u_\ep\|_N^N =1$, we get
\begin{equation}\label{eq:baab}
\int_{B_\de} |\na u_\ep|^N dx = 1-\frac1{c_\ep^{\frac N{N-1}}}\lt(-\frac N{\beta_N} \ln \de + A_\al + \al \|G_\al\|_N^N + o_\ep(\de) + o_\de(1)\rt).
\end{equation}
Denote $\tau_{\ep,\de} = \|\na u_\ep\|_{N,B_\de}^{\frac N{N-1}}$, and $u_{\ep,\de} =u_\ep -u_\ep(\de)$. Obviously, $u_{\ep,\de} \in W^{1,N}_0(B_\de)$. Using the result of Carleson and Chang \cite{CC1986}, we have
\begin{equation}\label{eq:CarlesonChang}
\limsup_{\ep\to 0} \int_{B_\de} \lt(e^{\beta_N u_{\ep,\de}^{\frac N{N-1}}/\tau_{\ep,\de}} -1\rt) dx \leq \de^N \frac{\om_{N-1}}{N} e^{1+ \frac12+ \cdots + \frac1{N-1}}.
\end{equation}
For a fixed $R >0$, we have $B_{Rr_\ep} \subset B_\de$ for $\ep$ small. We know from Lemma \ref{limpsiep} that $u_\ep = c_\ep(1+ o_\ep(R))$ uniformly on $B_{Rr_\ep}$. From \eqref{eq:baab}, we have $\tau_{\ep,\de} \leq 1$ for $\ep$ and $\de >0$ small. This together Lemma \ref{uppersup} and Lemma \ref{Greenfunction} leads to on $B_{Rr_\ep}$ that
\begin{align*}
\beta_\ep (1+ \al\|u_\ep\|_N^N)^{\frac1{N-1}}u_\ep^{\frac N{N-1}}&\leq \beta_N u_{\ep}^{\frac{N}{N-1}} + \frac{\al\beta_N}{N-1} \|u_\ep\|_N^N  u_{\ep}^{\frac N{N-1}}\\
& = \beta_N \lt(u_{\ep,\de} + u_\ep(\de)\rt)^{\frac N{N-1}} + \frac{\al \beta_N}{N-1} \|G_\al\|_N^N + o_\ep(R) \\
&=\beta_N u_{\ep,\de}^{\frac{N}{N-1}} + \frac{N\beta_N}{N-1}u_{\ep,\de}^{\frac1{N-1}} u_\ep(\de) + \frac{\al \beta_N}{N-1} \|G_\al\|_N^N + o_\ep(R)  + o_\ep(\de)\\
&= \frac{u_{\ep,\de}^{\frac{N}{N-1}}}{\tau_{\ep,\de}} -N\ln \de +\beta_N A_\al + o_\ep(\de) + o_\ep(R)+ o_\de(1).
\end{align*}
Since $\Phi(t) \leq e^t$ with $t\geq 0$, then
\begin{align*}
\int_{B_{Rr_\ep}} \Phi_N(\beta_\ep (1+ \al&\|u_\ep\|_N^N)^{\frac1{N-1}} u_\ep^{\frac N{N-1}}) dx\notag\\
&\leq \de^{-N} e^{\beta_N A_\al +o_\ep(\de) + o_\ep(R)+ o_\de(1)}\int_{B_{Rr_\ep}} e^{\beta_N u_{\ep,\de}^{\frac N{N-1}}/\tau_{\ep,\de}} dx\notag\\
&=\de^{-N} e^{\beta_N A_\al +o_\ep(\de) + o_\ep(R)+ o_\de(1)}\int_{B_{Rr_\ep}}\lt( e^{\beta_N u_{\ep,\de}^{\frac N{N-1}}/\tau_{\ep,\de}}-1\rt) dx + o_\ep(\de)\notag\\
&\leq \de^{-N} e^{\beta_N A_\al +o_\ep(\de) + o_\ep(R)+ o_\de(1)}\int_{B_\de}\lt( e^{\beta_N u_{\ep,\de}^{\frac N{N-1}}/\tau_{\ep,\de}}-1\rt) dx + o_\ep(\de).
\end{align*}
This together \eqref{eq:CarlesonChang} implies by letting $\ep \to 0$ and $\de\to 0$
\begin{equation}\label{eq:x1}
\limsup_{\ep \to 0} \int_{B_{Rr_\ep}} \Phi_N(\beta_\ep (1+ \al\|u_\ep\|_N^N)^{\frac1{N-1}} u_\ep^{\frac N{N-1}}) dx \leq \frac{\om_{N-1}}{N} e^{\beta_N A_\al + 1+ \frac12+ \cdots + \frac1{N-1}}
\end{equation}
for any $R >0$. In the other hand, using \eqref{eq:vocungnho1} and making the change of variable, we have
\begin{align*}
\int_{B_{Rr_\ep}} \Phi_N(\beta_\ep (1+ \al\|u_\ep\|_N^N)^{\frac1{N-1}} &u_\ep^{\frac N{N-1}}) dx \\
&=(1+o_\ep(R))\int_{B_{Rr_\ep}} e^{\beta_\ep (1+ \al\|u_\ep\|_N^N)^{\frac1{N-1}} u_\ep^{\frac N{N-1}}} dx\\
&=(1+o_\ep(R))\frac{\lam_\ep}{c_\ep^{\frac{N}{N-1}}}\int_{B_{R}} e^{\beta_\ep (1+ \al\|u_\ep\|_N^N)^{\frac1{N-1}} c_\ep^{\frac{N}{N-1}}(\psi_\ep^{\frac N{N-1}}-1)} dx\\
&=(1+o_\ep(R))\frac{\lam_\ep}{c_\ep^{\frac{N}{N-1}}}\int_{B_{R}} e^{(1+o_{\ep}(R))\beta_N \frac{N}{N-1}\vphi} dx,
\end{align*}
here the last equality comes from \eqref{eq:uniformvphi}. Letting $\ep \to 0$ and $R\to \infty$ we obtain \eqref{eq:chantrenMT} by using \eqref{eq:x1} and \eqref{eq:uppersup}. Our proof is finished.
\end{proof}

To finish the proof of Theorem \ref{Maintheorem} in the critical case, we will construct a sequence of test functions $\phi_\ep$ such that $\|\phi_\ep\|_{W^{1,N}(\R^N)} =1$ such that
\begin{equation}\label{eq:test}
\int_{\R^N } \Phi_N(\beta_\ep (1+ \al\|\phi_\ep\|_N^N)^{\frac1{N-1}} |\phi_\ep|_\ep^{\frac N{N-1}}) dx  > \frac{\om_{N-1}}{N} e^{\beta_N A_\al + 1+ \frac12+ \cdots + \frac1{N-1}},
\end{equation}
holds for $\al, \ep >0$ small.

For this purpose, let us define
\begin{equation}\label{eq:phiep}
\phi_\ep(x) =
\begin{cases}
c + \frac{1}{c^{\frac 1{N-1}}}\lt(-\frac{N-1}{\beta_N}\ln\lt(1 + c_N \lt(\frac{|x|}\ep\rt)^{\frac{N}{N-1}}\rt) + A\rt), &\mbox{$|x| \leq R\ep$,}\\
\frac{1}{c^{\frac1{N-1}}} G_\al&\mbox{$|x| > R\ep$,}
\end{cases}
\end{equation}
where $c_N = \beta_N/ N^{N/(N-1)} = (\om_{N-1}/ N)^{1/(N-1)}$, $R = -\ln \ep$ and $b,c$ are constants determined later. To ensure $\phi_\ep \in W^{1,N}(\R^N)$ we must have
\[
c + \frac{1}{c^{\frac 1{N-1}}}\lt(-\frac{N-1}{\beta_N}\ln\lt(1 + c_N R^{\frac{N}{N-1}}\rt) + A\rt) = \frac1{c^{\frac1{N-1}}} G_\al(R\ep).
\]
Using the form \eqref{eq:Greenform} of $G_\al$ and $R = -\ln \ep$, we get
\begin{equation}\label{eq:caac}
c^{\frac{N}{N-1}} = \frac1{\beta_N} \ln \frac{\om_{N-1}}N +A_\al -\frac{N}{\beta_N} \ln \ep -A + O(R^{-\frac{N}{N-1}}).
\end{equation}
Integration by parts together \eqref{eq:Green} and \eqref{eq:Greenform} gives
\begin{equation}\label{eq:out}
\int_{B_{R\ep}^c}(|\na \phi_\ep|^N + |\phi_\ep|^N)dx = \frac1{c^{\frac N{N-1}}}\lt[\al \|G_\al\|_N^N -\frac{N}{\beta_N} \ln R\ep + A_\al + O\lt(\lt(-R\ep \ln R\ep\rt)^N\rt)\rt].
\end{equation}
A straightforward computation shows that
\begin{equation}\label{eq:int1}
\int_{B_{R\ep}} |\na \phi_\ep|^N dx = \frac1{c^{\frac N{N-1}}} \lt(\frac{N}{\beta_N} \ln R + \frac1{\beta_N} \ln \frac{\om_{N-1}}N -\frac{N-1}{\beta_N} \sum_{k=1}^{N-1} \frac1k + O(R^{-\frac N{N-1}})\rt).
\end{equation}
Using \eqref{eq:caac}, we can easily check that
\begin{equation}\label{eq:int2}
\int_{B_{R\ep}} |\phi_\ep|^N dx = \frac1{c^{\frac N{N-1}}} O((-R\ep \ln (R\ep))^N).
\end{equation}
Putting \eqref{eq:out}, \eqref{eq:int1} and \eqref{eq:int2} together, we get 
\begin{equation}\label{eq:wwww}
\|\phi_\ep\|_N^N =\frac1{c^{\frac N{N-1}}} \lt( \|G_\al\|_N^N + O((-R\ep \ln (R\ep))^N)\rt)
\end{equation}
and 
\begin{align*}
&\|\phi_\ep\|_{W^{1,N}(\R^N)}^N \\
&\qquad = \frac1{c^{\frac N{N-1}}} \lt(\al \|G_\al\|_N^N -\frac N{\beta_N} \ln \ep + A_\al + \frac1{\beta_N} \ln \frac{\om_{N-1}}N -\frac{N-1}{\beta_N} \sum_{k=1}^{N-1} \frac1k + O(R^{-\frac N{N-1}})\rt).
\end{align*}
Hence, for $\ep$ small, we can choose $c$ such that $\|\phi_\ep\|_{W^{1,N}(\R^N)} =1$, and 
\begin{equation}\label{eq:c}
c^{\frac N{N-1}} = \al \|G_\al\|_N^N -\frac N{\beta_N} \ln \ep + A_\al + \frac1{\beta_N} \ln \frac{\om_{N-1}}N -\frac{N-1}{\beta_N} \sum_{k=1}^{N-1} \frac1k + O(R^{-\frac N{N-1}}). 
\end{equation}
From \eqref{eq:wwww}, \eqref{eq:c} and $R = -\ln \ep$, we get
\begin{align}\label{eq:yyyy}
(1+ \al \|\phi_\ep\|_N^N)^{\frac1{N-1}} &= 1+  \frac{\al \|G_\al\|_N^N}{(N-1)c^{\frac N{N-1}}} - \frac{N-2}{2(N-1)^2} \frac{\alpha^2\|G\|_N^{2N}}{c^{\frac{2N}{N-1}}} + O(c^{-\frac{3N}{N-1}}).
\end{align}

We continue by estimating the integral of $\Phi_N(\beta_N(1+ \al\|\phi_\ep\|_N^N)^{\frac1{N-1}} |\phi_\ep|^{\frac N{N-1}})$ on $\R^N$. Using the inequality $\Phi_N(t) \geq t^{N-1}/(N-1)!$, $t\geq 0$, we have
\begin{align}\label{eq:out*}
\int_{B_{R\ep}^c} \Phi_N(\beta_N(1+& \al\|\phi_\ep\|_N^N)^{\frac1{N-1}} |\phi_\ep|^{\frac N{N-1}}) dx \notag\\
&\geq \frac{\beta_N^{N-1}(1+ \al\|\phi_\ep\|_N^N)}{(N-1)!} \int_{B_{R\ep}^c} |\phi_\ep|^N dx\notag\\
&= \frac{\beta_N^{N-1}}{(N-1)!} \frac{(1+ \al\|\phi_\ep\|_N^N)\lt(\|G_\al\|_N^N + O((-R\ep \ln (R\ep))^N)\rt)}{c^{\frac N{N-1}}}\notag\\
&=\frac{\beta_N^{N-1}}{(N-1)!}\frac{\|G_\al\|_N^N}{c^{\frac N{N-1}}} + O(R^{-\frac N{N-1}}),
\end{align}
here we use \eqref{eq:wwww} and \eqref{eq:c} and $R =-\ln \ep$. 

For $\ep >0$ small, we have $\phi_\ep > 0$ on $B_{Rr_\ep}$. It is easy to check that $(1+t)^a \geq 1 + at$ for any $t >-1$ and $a \in (1,2]$. Applying this inequality, we get on $B_{R\ep}$
\begin{align*}
\phi_\ep^{\frac N{N-1}} &\geq c^{\frac N{N-1}} + \frac{N}{N-1}\lt(-\frac{N-1}{\beta_N}\ln\lt(1 + c_N \lt(\frac{|x|}\ep\rt)^{\frac{N}{N-1}}\rt) + A\rt).
\end{align*}
Plugging \eqref{eq:caac} and \eqref{eq:c} into this estimate, we get
\begin{align*}
\phi_\ep^{\frac N{N-1}}&\geq -\frac{\al}{N-1} \|G_\al\|_N^N + A_\al+ \frac1{\beta_N} \ln \frac{\om_{N-1}}N +\frac{1}{\beta_N} \sum_{k=1}^{N-1} \frac1k  + \frac{N}{\beta_N} R + \frac{N}{N-1} \vphi\lt(\frac x\ep\rt)\\
&=c^{\frac{N}{N-1}} -\frac{N}{N-1} \al \|G_\al\|_N^N + \frac{N}{\beta_N} \sum_{k=1}^{N-1} \frac1k + \frac N{N-1} \vphi\lt(\frac x\ep\rt),
\end{align*}
here the equality is obtained by using again \eqref{eq:c}. This together \eqref{eq:yyyy} leads to
\begin{align}\label{eq:fiwi}
(1+ \al \|\phi_\ep\|_N^N)^{\frac1{N-1}}\phi_\ep^{\frac N{N-1}}&\geq -\frac N{\beta_N} \ln \ep + A_\al + \frac1{\beta_N} \ln \frac{\om_{N-1}}N + \frac1{\beta_N} \sum_{k=1}^{N-1} \frac1k\notag\\
&\quad + \frac N{N-1} \lt(\sum_{k=1}^{N-1} \frac1k\rt) \frac{\al \|G_\al\|_N^{N}}{c^{\frac N{N-1}}} - \frac{N^2-N+2}{2(N-1)^2} \frac{\al^2 \|G_al\|_N^{2N}}{c^{\frac N{N-1}}}\notag\\
&\quad +  \lt(1+  \frac{\al \|G_\al\|_N^N}{(N-1)c^{\frac N{N-1}}}\rt) \frac N{N-1} \vphi\lt(\frac x\ep\rt) 
\end{align}
From the definition of $\phi_\ep$ and \eqref{eq:caac}, we have
\begin{equation}\label{eq:wifi}
\int_{B_{R\ep}} \Phi_N(\beta_N(1+ \al \|\phi_\ep\|_N^N)^{\frac1{N-1}}\phi_\ep^{\frac N{N-1}}) dx = \int_{B_{R\ep}} e^{\beta_N(1+ \al \|\phi_\ep\|_N^N)^{\frac1{N-1}}\phi_\ep^{\frac N{N-1}}} dx + O(R^{-\frac{N}{N-1}}).
\end{equation}
Plugging \eqref{eq:fiwi} into \eqref{eq:wifi} and using \eqref{eq:vphian}, we get
\begin{align}\label{eq:jqk}
\int_{B_{R\ep}} \Phi_N(\beta_N(1+ \al &\|\phi_\ep\|_N^N)^{\frac1{N-1}}\phi_\ep^{\frac N{N-1}}) dx \notag\\
&\geq \frac{\om_{N-1}}N e^{\beta_N A_\al +\sum_{k=1}^{N-1} \frac1k +\frac {N\beta_N}{N-1} \lt(\sum_{k=1}^{N-1} \frac1k\rt) \frac{\al \|G\|_N^{N}}{c^{\frac N{N-1}}} - \frac{(N^2-N+2)\beta_N}{2(N-1)^2} \frac{\al^2 \|G\|_N^{2N}}{c^{\frac N{N-1}}}}\times\notag\\
&\quad \times  \int_{B_{R}} \lt(1 + c_N |x|^{\frac{N}{N-1}} \rt)^{-N\lt(1+  \frac{\al \|G_\al\|_N^N}{(N-1)c^{\frac N{N-1}}}\rt)} dx + O(R^{-\frac N{N-1}}).
\end{align}
It is easy to check that 
\begin{align*}
\int_{B_{R}} \Big(1 + c_N |x|^{\frac{N}{N-1}} &\Big)^{-N\Big(1+ \frac{\al \|G_\al\|_N^N}{(N-1)c^{\frac N{N-1}}}\Big)} dx\\
& = \int_{\R^N} \lt(1 + c_N |x|^{\frac{N}{N-1}} \rt)^{-N\lt(1+  \frac{\al \|G_\al\|_N^N}{(N-1)c^{\frac N{N-1}}}\rt)} dx + O(R^{-\frac N{N-1}})
\end{align*}
and 
\begin{align*}
\int_{\R^N} \Big(1 + c_N |x|^{\frac{N}{N-1}} \Big)^{-N\Big(1+ \frac{\al \|G_\al\|_N^N}{(N-1)c^{\frac N{N-1}}}\Big)} dx& =\frac{\Gamma\lt(1 + \frac{N\al \|G_\al\|_N^N}{(N-1)c^{\frac N{N-1}}}\rt) \Gamma(N)}{\Gamma\lt(N + \frac{N\al \|G_\al\|_N^N}{(N-1)c^{\frac N{N-1}}}\rt)} \\
&= 1 -\frac{N^2}{2(N-1)!}\frac{\al \|G_\al\|_N^N}{c^{\frac N{N-1}}} + O(R^{-\frac N{N-1}}).
\end{align*}
These two estimates and \eqref{eq:jqk} imply
\begin{align}\label{eq:int*}
\int_{B_{R\ep}} \Phi_N(\beta_N(1+ \al &\|\phi_\ep\|_N^N)^{\frac1{N-1}}\phi_\ep^{\frac N{N-1}}) dx \notag\\
&\geq \frac{\om_{N-1}}N e^{\beta_N A_\al +\sum_{k=1}^{N-1} \frac1k} \Bigg[1 +\lt(\frac{N\beta_N}{N-1} \sum_{k=1}^{N-1}\frac1k -\frac{N^2}{2(N-1)!}\rt)\frac{\al \|G_\al\|_N^N}{c^{\frac N{N-1}}}\notag\\
&\hspace{3cm} -\frac{(N^2 -N+2)\beta_N}{2(N-1)^2}\frac{\al^2 \|G_\al\|_N^{2N}}{c^{\frac N{N-1}}} +O(R^{-\frac N{N-1}}) \Bigg]
\end{align}
Combining \eqref{eq:out*}, \eqref{eq:int*} and \eqref{eq:c}, we obtain
\begin{align}\label{eq:vediem}
\int_{\R^N} &\Phi_N(\beta_N(1+ \al \|\phi_\ep\|_N^N)^{\frac1{N-1}}\phi_\ep^{\frac N{N-1}}) dx\notag\\
&\geq \frac{\om_{N-1}}{N} e^{\beta_N A_\al +1 + \frac12 + \cdots+ \frac1{N-1}} + \frac{\|G_\al\|_N^N}{c^{\frac N{N-1}}}\Bigg[\frac{\beta_N^{N-1}}{(N-1)!} + \frac{\om_{N-1}}N e^{\beta_N A_\al +\sum_{k=1}^{N-1} \frac1k} \times \notag\\
&\qquad \times \alpha \Bigg(\frac{N\beta_N}{N-1} \sum_{k=1}^{N-1}\frac1k -\frac{N^2}{2(N-1)!} -\al \|G\|_N^{N}\frac{(N^2 -N+2)\beta_N}{2(N-1)^2}\Bigg) + o_\ep(1)\Bigg].
\end{align}
Since both $\|G_\al\|_N^N$ and $A_\al$ depend on $\al$. This makes the difficulties to determine the sign of the term in the bracket in \eqref{eq:vediem}. However, we have the following.
\begin{lemma}\label{almostdone}
There exists $\al_0 \in (0,1)$ such that
\begin{align*}
\frac{\beta_N^{N-1}}{(N-1)!} + \frac{\om_{N-1}}N &e^{\beta_N A_\al +\sum_{k=1}^{N-1} \frac1k}\alpha \, \times\\
&\times \Bigg(\frac{N\beta_N}{N-1} \sum_{k=1}^{N-1}\frac1k -\frac{N^2}{2(N-1)!} -\al \|G_\al\|_N^{N}\frac{(N^2 -N+2)\beta_N}{2(N-1)^2}\Bigg) > 0
\end{align*}
for any $0< \al < \al_0$.
\end{lemma}
\begin{proof}
Let $G_0$ be solution of \eqref{eq:Green} corresponding to $\al =0$, i.e., $G_0$ is the distributional solution of 
\[
-\Delta_N G_0 + G_0^{N-1} = \de_0.
\]
Note that 
\[
-\De_N G_\al + (1-\al)G_\al^{N-1} = \de_0.
\]
A simple scaling argument shows that $G_\al(x) = G_0((1-\al)^{1/N} x)$ which implies
\[
A_\al = A_0 -\frac1{\beta_N} \ln(1-\al),\qquad \|G_\al\|^N_N = (1-\al)^{-1} \|G_0\|_N^N.
\] 
Hence
\begin{align*}
& e^{\beta_N A_\al +\sum_{k=1}^{N-1} \frac1k}\alpha \Bigg(\frac{N\beta_N}{N-1} \sum_{k=1}^{N-1}\frac1k -\frac{N^2}{2(N-1)!} -\al \|G_\al\|_N^{N}\frac{(N^2 -N+2)\beta_N}{2(N-1)^2}\Bigg)\\
&\quad = e^{\beta_N A_0 +\sum_{k=1}^{N-1} \frac1k}\frac{\alpha}{1-\al} \Bigg(\frac{N\beta_N}{N-1} \sum_{k=1}^{N-1}\frac1k -\frac{N^2}{2(N-1)!} -\frac{\al}{1-\al} \|G_0\|_N^{N}\frac{(N^2 -N+2)\beta_N}{2(N-1)^2}\Bigg).
\end{align*}
By this equality, we easily obtain the conclusion of this lemma.
\end{proof}
Let $\al_0$ be the number in Lemma \ref{almostdone}. \eqref{eq:vediem} together Lemma \ref{almostdone} implies \eqref{eq:test} for any $0\leq \al < \alpha_0$ when $\ep >0$ small. 

We now have all ingredients to prove Theorem \ref{Maintheorem} in the critical case.\\

\begin{proof}[Proof of Theorem \ref{Maintheorem} in the critical case]
Let $\al_0 \in (0,1)$ as in Lemma \ref{almostdone}, we have \eqref{eq:test} for $\phi_\ep$ defined by \eqref{eq:phiep} and $\ep >0$ small enough. In the light of Lemma \ref{chantrenMT}, the sequence $c_\ep$ is bounded, and hence we can use Lemma \ref{boundedcase} to finish our proof.\\
\end{proof}


\section{Proof of Theorem \ref{Maintheorem2}}
This section is devoted to prove Theorem \ref{Maintheorem2}, that is, $MT(N,\beta_N,1) =\infty$. Our main tools are the Moser sequence used by Moser \cite{M1970} and by Adachi and Tanaka \cite{Adachi} together the scaling argument. We first recall an exchange of functions between function on $\R$ and the radial function on $\R^N$. Let $u$ be a radial function on $\R^N$, as before, we write $u(r)$ instead of $u(x)$ with $|x| =r$. Following the idea of Moser, we define a new function $w$ on $\R$ by
\begin{equation}\label{eq:relation}
w(t) = N^{1-\frac1N} \om_{N-1}^{\frac1N} u(e^{-\frac tN}).
\end{equation}
Then we have the following relation
\begin{equation}\label{eq:Phiofu}
\int_{\R^N} \Phi_N(\beta |u(x)|^{\frac N{N-1}}) dx = \frac{\om_{N-1}}N \int_{\R} \Phi_N\lt(\frac{\beta}{\beta_N} |w(t)|^{\frac N{N-1}}\rt)e^{-t} dt.
\end{equation}
For $k \geq 1$, we consider the sequence $u_k$ of functions in $W^{1,N}(\R^N)$ defined by
\begin{equation}\label{eq:uk}
u_k(x)= 
\begin{cases}
0&\mbox{if $|x|\geq 1$,}\\
-\om_{N-1}^{-\frac1N} \lt(\frac kN\rt)^{-\frac1N} \ln |x|&\mbox{if $e^{-\frac kN} \leq |x| < 1$,}\\
\om_{N-1}^{-\frac1N} \lt(\frac kN\rt)^{\frac{N-1}N}&\mbox{if $0\leq |x| < e^{-\frac kN}$.}
\end{cases}
\end{equation}
Note that $\|\na u_k\|_N =1$ and 
\begin{equation}\label{eq:normuk}
\int_{\R^N} u_k(x)^{N} dx = \frac{N!}{N^N k} + O(k^{N-1}e^{-k}).
\end{equation}
For $R >0$, define $u_{k,R}(x) = u_{k}(x/R)$, and $\tilde u_{k,R} = u_{k,R}/ \|u_{k,R}\|_{W^{1,N}(\R^N)}$. It is obvious that $\|\na u_{k,R}\|_N^N =1$ and by \eqref{eq:normuk} 
\[
\int_{\R^N} |u_{k,R}|^N dx = R^N \int_{\R^N} |u_k|^N dx = R^N\lt(\frac{N!}{N^N k} + O(k^{N-1}e^{-k})\rt),
\]
which implies
\begin{equation}\label{eq:tiemcan}
\frac{1+ \|\tilde u_{k,R}\|_N^N}{\|u_{k,R}\|_{W^{1,N}(\R^N)}^N}= \frac{1 +2 \|u_{k,R}\|_N^N}{(1+ 2\|u_{k,R}\|_N^N)^2} = 1+ O(\|u_{k,R}\|_N^{2N}) = 1 + R^N O(k^{-2}).
\end{equation}
Let $w_k$ be the function determined from $u_k$ by \eqref{eq:relation}, we then have
\begin{equation}\label{eq:wk}
w_k(t) = 
\begin{cases}
0 &\mbox{if $t \leq 0$,}\\
k^{\frac{N-1}N} \frac tk&\mbox{if $0< t \leq k$,}\\
k^{\frac{N-1}N} &\mbox{if $t > k$.}
\end{cases}
\end{equation}
Since $\|\tilde u_{k,R}\|_{W^{1,N}(\R^N)} =1$, from \eqref{eq:Phiofu}, \eqref{eq:tiemcan}, \eqref{eq:wk} and integration by parts, we get
\begin{align*}
MT(N,\beta_N,1)&\geq \int_{\R^N} \Phi_N(\beta_N(1+ \|\tilde u_{k,R}\|_N^N)^{\frac1{N-1}} \tilde u_{k,R}^{\frac N{N-1}}) dx\\
&= \int_{\R^N} \Phi_N\lt(\beta_N \lt(\frac{1+ \|\tilde u_{k,R}\|_N^N}{\|u_{k,R}\|_{W^{1,N}(\R^N)}^N}\rt)^{\frac1{N-1}} u_{k,R}^{\frac N{N-1}}\rt) dx\\
&=R^N\int_{\R^N} \Phi_N\lt(\beta_N \lt(\frac{1+ \|\tilde u_{k,R}\|_N^N}{\|u_{k,R}\|_{W^{1,N}(\R^N)}^N}\rt)^{\frac1{N-1}} u_{k}^{\frac N{N-1}}\rt) dx\\
&=R^N\frac{\om_{N-1}}{N} \int_{\R} \Phi_N\lt(\lt(\frac{1+ \|\tilde u_{k,R}\|_N^N}{\|u_{k,R}\|_{W^{1,N}(\R^N)}^N}\rt)^{\frac1{N-1}} w_k^{\frac N{N-1}}\rt) e^{-t} dt\\
&=R^N\frac{\om_{N-1}}{N} \int_{\R} \Phi_N\lt(\lt(1+ R^N O(k^{-2})\rt)^{\frac1{N-1}} w_k^{\frac N{N-1}}\rt) e^{-t} dt\\
&\geq R^N\frac{\om_{N-1}}{N} \int_{k}^\infty \Phi_N\lt(\lt(1+ R^N O(k^{-2})\rt) w_k^{\frac N{N-1}}\rt) e^{-t} dt\\
&=R^N\frac{\om_{N-1}}{N} \Phi_N\lt(\lt(1+ R^N O(k^{-2})\rt) k\rt) e^{-k}\\
&=R^N\frac{\om_{N-1}}{N} e^{-k}\lt(e^{k + R^{N}O(k^{-1})} -\sum_{j=0}^{N-2} \frac{(k+ R^{N} O(k^{-1}))^j}{j!}\rt).
\end{align*}
Let $k\to \infty$ we obtain
\[
MT(N,\beta_N,1) \geq R^{N} \frac{\om_{N-1}}{N},
\]
for any $R >0$, hence $MT(N,\beta_N,1) =\infty$. This finishes our proof.

\section*{Acknowledgments}
This work was supported by the CIMI's postdoctoral research fellowship.


\begin{thebibliography}{99}
\bibitem{Adachi}
S. Adachi, and K. Tanaka, \emph{Trudinger type inequalities in $\R^N$ and their best constant\text}, Proc. Amer. Math. Soc., {\bf 128} (2000) 2051--2057.

\bibitem{Adams1988}
D. R. Adams, \emph{A sharp inequality of J. Moser for higher order derivatives\text}, Ann. of Math., {\bf 128} (2) (1988) 385-398.

\bibitem{AD2004}
Adimurthi, and O. Druet, \emph{Blow-up analysis in dimension $2$ and a sharp form of Trudinger--Moser inequality\text}, Comm. Partial Differ. Equ., {\bf 29} (2004) 295--322.

\bibitem{AY10}
Adimurthi, and Y. Yang, \emph{An interpolation of Hardy inequality and Trudinger--Moser inequality in $\R^N$ and its applications\text}, Int. Math. Res. Not., IMRN {\bf 13} (2010) 2394--2426.

\bibitem{Beckner}
W. Beckner, \emph{Estimates on Moser embedding\text}, Potential Anal., {\bf 20} (2004) 345--359.



\bibitem{Brothers}
J. E. Brothers, and W. P. Ziemer, \emph{Minimal rearrangements of Sobolev functions\text}, J. Reine. Angew. Math., {\bf 348} (1988) 153--179.

\bibitem{Cao}
D. Cao, \emph{Nontrivial solution of semilinear elliptic equations with critical exponent in $\R^2$\text}, Comm. Partial Differential Equations, {\bf 17} (1992) 407--435.

\bibitem{CC1986}
L. Carleson, and S. Y. A. Chang, \emph{On the existence of an extremal function for an inequality of J. Moser\text}, Bull. Sci. Math., {\bf 110} (1986) 113-127.

\bibitem{ChenLi}
W. Chen, and C. Li, \emph{Classification of solution of some nonlinear elliptic equations\text}, Duke. Math. J., {\bf 63} (1991) 615--622.

\bibitem{CLhei}
W. S. Cohn, and G. Lu, \emph{Best constants for Moser-Trudinger inequalities on the Heisenberg group\text}, Indiana Univ. Math. J., {\bf 50} (2001) 1567-1591.

\bibitem{CLcomplex}
W. S. Cohn, and G. Lu, \emph{Sharp constants for Moser-Trudinger inequalities on spheres in complex space $\C^n$\text}, Comm. Pure Appl. Math., {\bf 57} (2004) 1458-1493.

\bibitem{doO97}
J. M. do \'O, \emph{$N-$Laplacian equations in $\R^N$ with critical growth\text}, Abstr. Appl. Anal., {\bf 2} (1997) 301--315.

\bibitem{doO2014}
J. M. do \'O, and M. de Souza, E. de Medeiros, U. B. Severo, \emph{An improvement for the Trudinger--Moser inequality and applications\text}, J. Differential Equations, {\bf 256} (2014) 1317--1349.

\bibitem{doO2014*}
J. M. do \'O, and M. de Souza,  \emph{A sharp Trudinger--Moser type inequality in $\R^2$\text}, Trans. Amer. Math. Soc., {\bf 366} (2014) 4513--4549.

\bibitem{doO2015}
J. M. do \'O, and M. de Souza, \emph{A sharp inequality of Trudinger--Moser type and extremal functions in $H^{1,n}(\R^n)$\text}, J. Differential Equations, {\bf 258} (2015) 4062--4101.

\bibitem{doO2016}
J. M. do \'O, and M. de Souza, \emph{Trudinger--Moser inequality on the whole plane and extremal functions\text}, Commun. Contemp. Math., {\bf 18} (2016) 32 pp.

\bibitem{Druet}
O. Druet, E. Hebey, and F. Robert, \emph{Blow-up theory for elliptic PDEs in Riemannian geometry\text}, Math. Notes, vol. 45, Princeton University press, Princeton, NJ, 2004.

\bibitem{Esposito}
P. Esposito, \emph{A classification result for the quasi-linear Liouville equation\text}, preprint, arXiv.1609.03608.

\bibitem{Flucher1992}
M. Flucher, \emph{Extremal functions for the Trudinger-Moser inequality in $2$ dimensions\text} 
Comment. Math. Helv., {\bf 67} (1992) 471--497

\bibitem{Fontana1993}
L. Fontana, \emph{Sharp borderline Sobolev inequalities on compact Riemannian manifolds\text}, Comment. Math. Helv., {\bf 68} (1993) 415--454.





\bibitem{Ishi}
M. Ishiwara, \emph{Existence and nonexistence of maximizers for variational problems associated with Trudinger--Moser inequalities in $\mathbb R^N$\text}, Math. Ann., {\bf 351} (2011) 781--804.


\bibitem{Kichenassamy}
S. Kichenassamy, and L. Veron, \emph{Singular solution of the $p-$Laplace equation\text}, Math. Ann., {\bf 275} (1986) 599--615.

\bibitem{LamLuhei}
N. Lam, and G. Lu, \emph{Sharp Moser--Trudinger inequality on the Heisenberg group at the critical case and applications\text}, Adv. Math., {\bf 231} (2012) 3259--3287.



\bibitem{Li2001}
Y. Li, \emph{Moser--Trudinger inequaity on compact Riemannian manifolds of dimension two\text}, J. Partial Differ. Equa., {\bf 14} (2001) 163-192.

\bibitem{Li2005}
Y. Li, \emph{Extremal functions for the Moser-Trudinger inequalities on compact Riemannian manifolds\text}, Sci. China Ser. A, {\bf 48} (2005) 618--648.


\bibitem{LiRuf2008}
Y. Li, and B. Ruf, \emph{A sharp Trudinger-Moser type inequality for unbounded domains in $\R^n$\text}, Indiana Univ. Math. J., {\bf 57} (2008) 451--480.


\bibitem{LY2016}
X. Li, and Y. Yang, \emph{Extremal functions for singular Trudinger--Moser inequalities in the entire Euclidean space\text}, preprint, arXiv.1612.08247.



\bibitem{Lin1996}
K. Lin, \emph{Extremal functions for Moser's inequality\text}, Trans. Amer. Math. Soc., {\bf 348} (1996) 2663--2671.

\bibitem{Lions1985}
P. L. Lions, \emph{The concentration-compactness principle in the calculus of variations. The limit case. II\text}, Rev. Mat. Iberoamericana, {\bf 1} (1985) 45-121.

\bibitem{LuZhu}
G. Lu, and M. Zhu, \emph{A sharp Moser--Trudinger type inequality involving $L^n$ norm in the entire space $\R^n$\text}, preprint, arXiv:1703.00901v1.

\bibitem{M1970}
J. Moser, \emph{A sharp form of an inequality by N. Trudinger\text}, Indiana Univ. Math. J., {\bf 20} (1970/71) 1077--1092.

\bibitem{Nguyen2016}
V. H. Nguyen, \emph{An improvement for the sharp Adams  inequalities in bounded domains and whole space $\R^n$\text}, preprint, arXiv:1604.07526.


\bibitem{P1965}
S. I. Poho${\rm \check{z}}$aev, \emph{On the eigenfunctions of the equation $\Delta u + \lambda f(u) = 0$\text}, (Russian), Dokl. Akad. Nauk. SSSR, {\bf 165} (1965) 36-39.

\bibitem{Ruf2005}
B. Ruf, \emph{A sharp Trudinger-Moser type inequality for unbounded domains in $\R^2$\text}, J. Funct. Anal., {\bf 219} (2005) 340--367.


\bibitem{Serrin64}
J. Serrin, \emph{Local behavior of solutions of quasi-linear equations\text}, Acta. Math., {\bf 111} (1964) 248--302.


\bibitem{Struwe}
M. Struwe, \emph{Critical points of embeddings of $H_0^{1,n}$ into Orlicz spaces\text}, Ann. Inst. H. Poincar\'e, Analyse Non Lin\'eaire, {\bf 5} (1988) 425--464.

\bibitem{Struwe00}
M. Struwe, \emph{Positive solution of critical semilinear elliptic equations on non-contractive planar domain\text}, J. Eur. Math. Soc., {\bf 2} (2000) 329--388.

\bibitem{Tintarev}
C. Tintarev, \emph{Trudinger--Moser inequality with remainder terms\text}, J. Funct. Anal., {\bf 266} (2014) 55--66.

\bibitem{Tolksdorf}
P. Tolksdorf, \emph{Regularity for a more general class of quasilinear elliptics equations\text}, J. Differential Equations, {\bf 51} (1984) 126--154.

\bibitem{T1967}
N. S. Trudinger, \emph{On imbedding into Orlicz spaces and some applications\text}, J. Math. Mech., {\bf 17} (1967) 473-483.

\bibitem{WY2012}
G. Wang, and D. Ye, \emph{A Hardy--Moser--Trudinger inequality\text}, Adv. Math., {\bf 230} (2012) 294--320.


\bibitem{Weinstein}
I. M. Weinstein, \emph{Nonlinear Schr\"odinger equations and sharp interpolation estimates\text}, Commun. Math. Phys., {\bf 87} (1982/1983) 567--576.

\bibitem{Yang06}
Y. Yang, \emph{Extremal functions for Moser--Trudinger inequalities on $2-$dimensional compact Riemannian manifolds with boundary\text}, Int. J. Math., {\bf 17} (2006) 313--330.

\bibitem{Yang06*}
Y. Yang, \emph{A sharp form of Moser--Trudinger inequality in high dimension\text}, J. Funct. Anal., {\bf 239} (2006) 100--126.


\bibitem{Yang07corrige}
Y. Yang, \emph{Corrigendum to "A sharp form of Moser--Trudinger inequality in high dimension" [J. Funct. Anal., {\bf 239} (2006) 100--126]\text}, J. Funct. Anal., {\bf 242} (2007) 669--671.

\bibitem{Yang09}
Y. Yang, \emph{A sharp form of Moser--Trudinger inequality on compact Riemannian surface\text}, Trans. Amer. Math. Soc., {\bf 359} (2007) 5761--5776.

\bibitem{Yang2015}
Y. Yang, \emph{Extremal functions for Trudinger--Moser inequalities of Adimurthi--Druet type in dimension two\text}, J. Differential Equations, {\bf 258} (2015) 3161--3193.

\bibitem{YZ}
Y. Yang, and X. Zhu, \emph{Blow-up analysis concerning singular Trudinger--Moser inequalities in dimension two\text}, J. Funct. Anal., {\bf 272} (2017) 3347--3374.

\bibitem{Y1961}
V. I. Yudovi${\rm \check{c}}$, \emph{Some estimates connected with integral operators and with solutions of elliptic equations\text}, (Russian), Dokl. Akad. Nauk. SSSR, {\bf 138} (1961) 805-808.
\end{thebibliography}
\end{document}